\documentclass[sn-mathphys]{sn-jnl}

\usepackage{amsthm,amsmath,amssymb}
\usepackage{epsfig}
\usepackage{graphicx}
\usepackage{float}
\usepackage{caption}
\usepackage{subfigure}
\usepackage{enumerate}
\usepackage{epstopdf,color}
\usepackage{multirow}

\usepackage[marginal]{footmisc}

\newcommand{\bbE}{{\bf E}}

\newcommand{\bbT}{{\bf T}}

\newcommand{\bbU}{{\bf U}}
\newcommand{\bbV}{{\bf V}}

\newcommand{\um}{{\underline m}}
\newcommand{\ug}{{\underline g}}
\newcommand{\uF}{{\underline F}}
\newcommand{\rank}{{\text{rank}}}
\newcommand{\rVar}{{\text{Var}}}

\newcommand{\ep}{\varepsilon}
\newcommand{\rtr}{\mathrm{tr}}
\newcommand{\gs}{\sigma}

\newcommand{\bqn}{\begin{eqnarray*}}
	\newcommand{\eqn}{\end{eqnarray*}}
\newcommand{\bqa}{\begin{eqnarray}}
\newcommand{\eqa}{\end{eqnarray}}

\newtheorem{thm}{Theorem}[section]

\newtheorem{lemma}[thm]{Lemma}

\allowdisplaybreaks[4] 

\theoremstyle{thmstyleone}%

\jyear{2022}%

\theoremstyle{thmstyleone}%
\newtheorem{theorem}{Theorem}
\newtheorem{proposition}[theorem]{Proposition}%

\theoremstyle{thmstyletwo}%
\newtheorem{remark}{Remark}%

\theoremstyle{thmstylethree}%

\begin{document}

\title[Analysis of the limiting spectral distribution]{Analysis of the limiting spectral distribution of large dimensional general information-plus-noise type matrices}

\author{ \sur{Huanchao Zhou}}

\author*{\sur{Jiang Hu}*}\email{huj156@nenu.edu.cn}

\author{ \sur{Zhidong Bai}}

\author{ \sur{Jack W. Silverstein}}

\affil[1]{\orgdiv{School of Mathematics and Statistics}, \orgname{Northeast Normal University}, 
	\country{Changchun, China}}

\affil*[2]{\orgdiv{School of Mathematics and Statistics}, \orgname{Northeast Normal University}, 
	\country{Changchun, China}}
\affil[3]{\orgdiv{School of Mathematics and Statistics}, \orgname{Northeast Normal University}, 
	\country{Changchun, China}}

\affil[4]{\orgdiv{Department of Mathematics}, \orgname{North Carolina State University}, 
	\country{USA}}

\abstract{In this paper, we derive the analytical behavior of the limiting spectral distribution of non-central covariance matrices of the ``general
	information-plus-noise" type, as studied in \cite{zhou2022limiting}. Through the equation defining its Stieltjes transform, it is shown 
	that the limiting distribution has a continuous derivative away from zero, the derivative being analytic wherever it is positive, and we show 
	the determination criterion for its support.  We also extend the result in \cite{zhou2022limiting} to allow for all possible ratios of row to 
	column of the underlying random matrix.	}




\maketitle

\section{Introduction}\label{sec1}

The key point considered in this paper is to characterize the limiting spectral distribution (LSD) of non-central covariance matrices of the form:
\begin{equation}
B_n=\frac1N(R_n+T^{1/2}X_n)(R_n+T^{1/2}X_n)^*,\label{bn}
\end{equation}
where $ ^* $ stands for the complex conjugate transpose, $X_n$ is an $n\times N$ random matrix with independent standardized entries, $T_n^{1/2}$,an $n\times n$ matrix, is the nonnegative definite square root of the $n\times n$ nonrandom, nonnegative definite matrix $T_n$, $R_n$, an $n\times N$  nonrandom matrix, and $n,N$ are such that, as  $\min(n,N)\to\infty$, $c_n\equiv n/N\to c\in(0,1]$.

Let us commence with some background information.  
It is well known that the LSD for random matrices has constituted a basic part of large dimensional random matrices theory (LDRMT).
Since the famous semicircular law and M-P law were established in \cite{wigner1958distribution} and \cite{marchenko1967eigenvalue} respectively, many researchers have contributed to its subsequent development. 
One of the most extensively investigated in LDRMT is the so-called sample covariance type matrix taking the form $ M_{n} =\frac{1}{N} T^{\frac{1}{2}}_{n} X_{n} X^{*}_{n}T^{\frac{1}{2}}_{n} $, where initially the entries of $ X_{n} $ are independent and identically distributed(i.i.d.).   
Results on the LSD of the sample covariance type matrix can be found in \cite{silverstein1995strong,silverstein1995analysis,yin1986limiting}.

The majority of known results for the sample covariance type matrix are under the central condition, that is, the entries of $ X_{n} $ are zero mean (which can be extended to allowing the entries to have a common mean).  
In fact, the large non-central random matrices also have significant implications that may be considered as an extension of non-central Wishart matrices, an important random matrix in multivariate linear regression under a non-null hypothesis. 
Owing to its applications in wireless communication and signal processing, it is also referred to as an information-plus-noise matrix.
The first work on this direction was done by Dozier and Silverstein \cite{dozier2007empirical}. The model they considered is $$ D_n =\frac{1}{N}(R_n + \sigma X_n)(R_n + \sigma X_n)^*, $$
the matrix $ R_n $ contains the information,  and the matrix $ \sigma X_n $
is considered as additive noise. 
They studied the LSD of $ D_n $ and further investigated the analytical properties of the LSD, as well as the characterization of its support set in \cite{dozier2007analysis}.
Loubaton and Vallet \cite{loubaton2011almost} considered the model $ D_{n} $.
They assume $ R_n $ is of fixed rank and $ X_n $ with means zero and covariances $ \frac{1}{N} $, and they characterized the a.s. limits of the few largest eigenvalues of $ D_n $.
Bai and Silverstein \cite{bai2012no} showed that, with additional assumptions on $R_n$, in any closed interval outside the support of the LSD of $ D_n $, with probability one, there are no eigenvalues of $ D_n $ falling in this interval for all large $ n $.  Capitaine \cite{capitaine2014exact} further proved the exact separation problem.

Observing that the noise $ \sigma X_n $ has a covariance matrix $\sigma^2I$ ($ I $ representing the $ n \times n $ identity matrix),  Zhou et al. \cite{zhou2022limiting} expanded the possibilities of the noise covariance matrix. They allowed the covariance matrix not only to be a multiple of identity, but also other non-negative definite matrices. This further extended Dozier and Silverstein's model $ D_n $ to information-plus-correlated-noise matrices 
$$C_n =\frac{1}{N}T_{n}^{\frac{1}{2}} \left(R_{n}+X_{n} \right) \left(R_{n} +X_{n} \right)^{*}T_{n}^{\frac{1} {2}},$$ and Zhou et al. \cite{zhou2022limiting} proved that, almost surely, the empirical spectral distribution (ESD), $F_n$, of the matrix $ C_n $, defined by $F_n(x)=(1/n)( \text{number of eigenvalues of} \ C_n \leq x)$ converges to a nonrandom LSD $ F $, under certain conditions. When $T_{n}$ is a central sample covariance matrix and independent of $X_{n}$, $C_n $ is referred to as the noncentral Fisher matrix \cite{ZhangB23,HouZ23}.
We refer to the canonical correlation analysis \cite{BaoH1,BaoH2} for an application of the high dimensional noncentral Fisher matrix.

We are now able to state the result in Zhou et al. \cite{zhou2022limiting}. 
Under the assumptions:
\begin{description}
	\item[(a)] $ c_n = n/N \to c \in (0,1] $, as $ \min\{n, N\} \to \infty $.
	\item[(b)] The entries of $ X_n $ are independent standardized random variables satisfying
	$$  \frac{1}{\eta^2 nN} \sum_{jk} \bbE(\vert x^{(n)}_{jk} \rvert^2I(\vert x^{(n)}_{jk} \vert  \geq \eta \sqrt{n})) \to 0. $$ 
	\item[(c)] $ T_n $ is $n\times n$ nonrandom nonnegative definite, $R_n$ is $n\times N$ nonrandom, with $T_n$ and  $ R_nR^*_n$ commutative.
	\item[(d)] As $ \min\{n, N\} \to \infty $, the two-dimensional distribution function $ H_n(s,t)$, 
	where $ H_n(s, t) = n^{-1} \sum_{i=1}^{n} I(s_i \le s, t_i \le t) $ converges weakly to a nonrandom probability distribution $ H(s, t) $ almost surely, where $ s_i $, $ t_i $ are the paired eigenvalues of  $\frac{1}{N}R_nR_n^*$ and $T_n$, respectively.
\end{description}
It is shown that, with probability one, for any $z\in\mathbb{C}^+$, any limit of a subsequence of the Stieltjes transform $m_{C_n}(z)$ must converge to a number $m\in\mathbb{C}^+$ satisfying 
the equations
\begin{equation}
	\left\{
	\begin{aligned}
		m=\int\frac{\mathrm{d}H(s,t)}{\frac{st}{1+cg}-(1+cmt)z+t(1-c)}, \\
		g=\int\frac{t\mathrm{d}H(s,t)}{\frac{st}{1+cg}-(1+cmt)z+t(1-c)}, \label{first}
	\end{aligned}
	\right.
\end{equation}	
where $g\in\mathbb{C}^+$.   We remark here that $g$ is the limit (when it exists) of $g_n=(1/n)\mathrm{tr}(C_n-zI)^{-1}T_n$.

Moreover, it is shown that for $c\leq 1$, there is only one pair $(m,g)$, each in $\mathbb{C}^+$ satisfying \eqref{first} for all $z\in\mathbb{C}^+$, 
which implies, with probability one, $F_n$ converges vaguely to a distribution $F$ having Stieltjes transform $m=m(z)$ satisfying  \eqref{first} for all
$z\in\mathbb{C}^+$.
We will presently prove that the convergence is actually weak, that is, $F$ is a probability distribution function, a fact omitted in Zhou et al. \cite{zhou2022limiting}.

Before continuing, it is necessary to introduce the major tool used in establishing almost sure weak convergence of ESDs.   
It is the Stieltjes transform of the empirical distribution of eigenvalues of a matrix.  
For any finite measure $\mu$ on $\mathbb{R}$ with distribution function $F$, the Stieltjes transform of $\mu$ is defined as $ z\in \mathbb{C}^{+}\equiv\{z\in \mathbb{C}: \Im z>0\}, $
\begin{equation}
\begin{split}
m_{\mu}(z)=m_{F}(z)= \int \dfrac{1}{\lambda-z} \mathrm{d}\mu(\lambda)= \int \dfrac{1}{\lambda-z} \mathrm{d}F(\lambda).\label{steiltdef}
\end{split}
\end{equation}
An important property is the inversion formula: $ F $ can be obtained by 
\begin{align*}
F(b)-F(a)=\frac{1}{\pi}\lim_{v\rightarrow 0^{+}}\int_{a}^{b}\Im m_{F}(x+iv)\mathrm{d}x, 
\end{align*}
where $ a $, $ b $ are continuity points of $ F $. 

The primary motivation for using the Stieltjes transform stems from the fact that for any $n\times n$ matrix $A$ with real eigenvalues, the Stieltjes transform of the ESD of $A$, denoted by $m_A(z)$, is given by 
$$m_A(z)=\frac1n\mathrm{tr}(A-zI)^{-1}.$$
Additional important properties of the Stieltjes transform can be explored in Lemma 2.2 of \cite{Shohat} and Theorems A.2, A.4, A.5 of \cite{Krein}:
If $f$ is analytic on $\mathbb{C}^+$, both $f(z)$ and $zf(z)$ map $\mathbb{C}^+$ into $\mathbb{C}^+$, and there is a $\theta\in(0,\pi/2)$ for which $zf(z)\to c$, finite, as $z\to\infty$ restricted to $\{ w\in\mathbb{C} : \theta< \arg w < \pi-\theta\}$, then $c$ is real, $c < 0$ and $f$ is the Stieltjes transform of a measure on the nonnegative reals with total mass $-c$.
Notice that any Stieltjes transform maps $\mathbb{C}^+$ to $\mathbb{C}^+$.   
Notice also that for any finite measure $\mu$ on the nonnegative reals, we also have $zm_{\mu}(z)$ maps $\mathbb{C}^+$ to $\mathbb{C}^+$ since with $z=x+iv$, 
$$\Im zm_{\mu}(z)=\int\frac{\lambda v}{\vert \lambda-z\vert^2}\mathrm{d}\mu(\lambda)>0.$$

Consider any sequence ${\mu_n}$ of probability measures. If it is established that there exists a countable number of $z_j\in\mathbb C^+$ with an accumulation point, such that the Stieltjes transform
$$m_{\mu_n}(z)\equiv\int\frac1{\lambda-z}\mathrm{d}\mu_n(\lambda)$$
converges for each of these $z_j$, then we can conclude that $\mu_n$ converges vaguely to a subprobability measure $\mu$, and weakly if $\mu$ is a probability measure.
Indeed, consider any vaguely convergent subsequence $\mu_{n_i}$, converging to $\hat\mu$. 
Given that both the real and imaginary parts of $1/(\lambda -z)$, viewed as functions of $\lambda$, are continuous and vanish at $\pm\infty$, it follows for these $z_j$,
$$\int\frac1{\lambda-z_j}\mathrm{d}\mu_{n_i}\to\int\frac1{\lambda-z_j}\mathrm{d}\hat\mu,\quad\text{ as }i\to\infty.$$
This limiting analytic function is uniquely determined by the values it takes on the $z_j$.  Therefore we have vague convergence of $\mu_n$ to a unique $\mu$, and the convergence is weak if $\mu$ is proper.   
   
Then we will state that the convergence is actually weak in Zhou et al. \cite{zhou2022limiting}.   
Let $U_n\Lambda_nU_n^*$ be the spectral decomposition of $C_n$. As mentioned previously, the quantity $g$ is the limit of 
$$g_n(z)=\frac1n\mathrm{tr}(C_n-zI)T_n=\frac1n \mathrm{tr} T_n^{1/2} U_n(\Lambda_n-zI)^{-1} U_n^*T_n^{1/2}=\frac1n\sum_{i=1}^n\frac{\vert T_n^{1/2}U_n\vert^2_{ii} }{\lambda_i-z},$$
where the $\lambda_i$ are the eigenvalues of $C_n$.  
We observe that $g_n$ is the Stieltjes transform of a measure, one that assigns mass to the eigenvalues of $C_n$ with a cumulative mass of $(1/n)\mathrm{tr} T_n$.   
Consequently, similar to $m(z)$, $g_n(z)$ also constitutes the Stieltjes transform of a finite measure on the nonnegative real line. It necessarily follows that the limit $g(z)$, if it exists, ensures both $\Im g(z)$ and $\Im zg(z)$ remain nonnegative.

It follows that, with $z=x+iv$, $v>0$, the absolute value of the imaginary part of the denominator of the integrand in both equations in \eqref{first} is greater than $v$, so that for fixed $x$ and all $v\ge v_0>0$, $z$ times the integrand in the first equation in \eqref{first} is bounded.  
For each $s,t$ as $v\uparrow\infty$ ($x$ fixed), $z$ times the integrand in the first equation of \eqref{first} converges to $-1$.  
Therefore, from the dominated convergence theorem, we have $zm(z)\to -1$ as $v\uparrow\infty$. This implies that the limiting $F$ is a probability distribution function, so that the convergence is weak.

We turn now to the aim of this paper, namely to investigate the limiting properties of the ensemble \eqref{bn}.
It is clear that $C_n$ is of this form, so it seems $B_n$ is a more general ensemble.  
However, it is shown in the appendix that the results on $C_n$ can be used to show the limiting behavior of the ESD of the eigenvalues of $B_n$. Restricting  $T_n$ to remain nonsingular, we can infer the limiting behavior.  
Indeed, write
\begin{equation}\label{bnwithinverse}B_n=\frac1NT_n^{1/2}(T_n^{-1/2}R_n+X_n)(T_n^{-1/2}R_n+X_n)^*T^{1/2},
\end{equation}
which is in the form of $C_n$ with the two dimensional distribution function
\begin{equation}\label{hnwithinverse}H'_n(s, t) = n^{-1} \sum_{i=1}^{n} I(t_i^{-1}s_i \le s, t_i \le t).
\end{equation}
Writing $H_n(s, t) = n^{-1} \sum_{i=1}^{n} I(s_i \le s, t_i \le t)$, if $H_n$ converges weakly to $H$ and if $H(\infty,0)=0$, then $H_n$ and $H'_n$ can be interpreted in terms of distribution functions of random variables $x_n$, $y_n$ where $(x_n,y_n)=(s_i,t_i)$ with probability $1/n$ and $H_n(s,t)=P(x_n\leq s,y_n\leq t)$, $H'_n(s,t)=P(y_n^{-1}x_n\leq s,y_n\leq t)$.  
We can think of $x_n,y_n$ converging in distribution to random variables $x$, $y$, with joint distribution function $H$.  
It follows that $y_n^{-1}x_n,y_n$ converge in distribution to random variables $y^{-1}x,y$, whose distribution function represents the weak limit of $H'_n$, referred to as $H'$.  
With $H'$ placed in the equations in \eqref{first}, the integrals can be interpreted as expected values of functions of $x$ and $y$. This leads us to the subsequent theorem:
\begin{theorem}
	Under assumptions (a) -- (d), for all $c\in(0,1]$, with probability one, the ESD of $\eqref{bn}$ converges weakly to $F$, nonrandom, with
	Stieltjes transform $m_F$ satisfying 
	
	\begin{equation}
		\left\{ 
		\begin{aligned}
			m=\int\frac{\mathrm{d}H(s,t)}{\frac{s}{1+cg}-(1+cmt)z+t(1-c)}, \\
			g=\int\frac{t\mathrm{d}H(s,t)}{\frac{s}{1+cg}-(1+cmt)z+t(1-c)}. \label{second}
		\end{aligned}\right.
	\end{equation}	

\end{theorem}

\begin{remark}
	From the proof, it is worth noting that the condition $c\leq 1$ is not necessary in this theorem. However, this restriction can guarantee the establishment of the following formula \eqref{Vlessthanz}, which is very useful in the proof of the existence of the density function, and we will give a remark later to illustrate the necessity of the existence of the condition. For the unity of conditions, $c\leq 1$ is required here.
\end{remark}

Using truncation techniques, the appendix will show how results in Zhou et al. \cite{zhou2022limiting} extend to $B_n$.  
It will also include a complete proof of almost sure weak convergence of the ESD for all positive $c$.  
This paper shows that, for $c\leq1$ and conditions on the limiting $H$, including the assumption that $H$ has bounded support,  the limiting distribution function $F$ has, for $x\neq0$ a continuous density, analytic in its support, along with a detailed analysis of how the support of $ F $ can be determined.  This is shown in the next section.

\section{Existence of a density}\label{sec2}

It is more convenient to consider the limiting ESD of the eigenvalues of the $N\times N$ matrix
\begin{equation}
	\underline{B}_n=\frac{1}{n}(R_n+T^{\frac{1}{2}}_n X_n)^*(R_n +  T^{\frac{1}{2}}_n X_n). \label{Bnul}
\end{equation}
The eigenvalues of $B_n$ are same as those of $\underline B_n$, except for $ \lvert n-N\rvert $ zero eigenvalues. 
It follows that their ESDs and Stieltjes transforms have  the  following relations
\begin{equation}
\begin{aligned}
F^{\underline{B}_n}=\left(1-\frac{n}{N} \right)I_{[0,\infty)}+\frac{n}{N} F^{ B_n}, \label{BBnul}
\\
m^{\underline{B}_n}(z)=-\frac{\left(1-\frac{n}{N} \right)}{z}+\frac{n}{N}m^{ B_n}. 
\end{aligned} 
\end{equation}
Then making a variable transformation
\begin{equation}
\begin{aligned}
&\underline{m}(z)=-\dfrac{1-c}{z}+cm(z),  \\ \label{vartrans}
&\underline{g}(z)=-\dfrac{1}{z(1+cg(z))},   
\end{aligned}
\end{equation}
where $ \underline{m}(z)$ is the Stieltjes transform of the LSD of $ F^{\underline{B}_n} $ .
Then the equations in \eqref{second} become
\begin{equation}
\begin{aligned}
z=-\dfrac{1-c}{\underline{m}}-\frac{c}{\underline{m}}\int \dfrac{\mathrm{d} H(s,t)}{1 + s\underline{g}(z) + t\underline{m}(z)},
\\
z=-\dfrac{1}{\underline{g}}+c\int \dfrac{t\mathrm{d} H(s,t)}{1 + s\underline{g}(z) + t\underline{m}(z)}. \label{third}
\end{aligned}
\end{equation}

Let $\underline F$ denote the almost sure limiting distribution function of the eigenvalues of $\underline B_n$, with Stieltjes transform 
$$m_{\underline F}(z)=\int\frac1{\lambda-z}\mathrm{d}\underline F(\lambda),\quad z\in\mathbb C^+.$$
We begin with deriving an important identity.
Write $  \underline{m}= \underline{m}_1+i\underline{m}_2 $, $ \underline{g} = \underline{g}_1+i\underline{g}_2 $, $ z\underline{m} = (z\underline{m})_1 + i(z\underline{m})_2 $, and $z\underline{g} = (z\underline{g})_1 + i(z\underline{g})_2$. 
Fix $z=x+iv\in\mathbb{C}^+$. 
Multiplying by $\underline m$ on both sides to the first equation of \eqref{third} and comparing the imaginary part of the 
resulting equation and that of the second equation, we obtain
$$  (z\underline{m})_2 = cA_1 \underline{g}_2 + cB_1 \underline{m}_2,  $$
$$  v =\frac{\underline{g}_2}{\vert \underline{g}\vert^2}-(c A_2 \underline{g}_2 + cB_2\underline{m}_2),  $$
where
\begin{align}
\begin{aligned}
A_j&=&\int \dfrac{st^{j-1}\mathrm{d} H(s,t)}{\vert 1 + s\underline{g} + t\underline{m}\vert^2}, j=1,2,
\\
B_j&=&\int \dfrac{t^{j}\mathrm{d} H(s,t)}{\vert 1 + s\underline{g} + t\underline{m}\vert^2}, j=0,1,2. \label{7216}
\end{aligned}
\end{align}
From this, we have
\begin{equation}\label{eqa0}
V_0\equiv \vert \underline{g}\vert^{-2}-cA_2>0,
\end{equation}
\begin{equation}
\underline{g}_2=V_0^{-1}(cB_2 \underline{m}_2+v), \nonumber
\end{equation}
and consequently
\begin{equation}
V \underline{m}_2+V_0^{-1} c A_1v =(z \underline{m})_2, \label{7218}
\end{equation}
where 
\begin{align}
\begin{aligned}
V\equiv c^2 A_1 B_2 V_0^{-1} + cB_1. \label{7219}
\end{aligned}
\end{align}

We now prove the following proposition.
\begin{proposition}\label{pro2}
	For $c<1$, $V=c^2 A_1 B_2 V_0^{-1} + cB_1<\vert z \vert .$
\end{proposition}
\begin{proof}[Proof of Proposition \ref{pro2}]
Multiplying by $\underline{m} $ and dividing by $ z $ on both sides of the first equation of \eqref{third} and 
comparing the imaginary parts we obtain
\begin{align}
\begin{aligned}
\underline{m}_2\vert z \vert^2 = v(1 - c) + cB_0v + cA_1 (z\underline{g})_2 + cB_1 (z\underline{m})_2. \label{7220}
\end{aligned}
\end{align}
Dividing $ z $ on both sides of the second equation of \eqref{third}, we have
\begin{align}
\begin{aligned}
0=(z\underline{g})_2 \vert \underline{g}\vert^{-2}-cB_1v-cA_2(z\underline{g})_2 -cB_2 (z\underline{m})_2.  \label{7221}
\end{aligned}
\end{align}
From \eqref{7221}, we have
$$ (z\underline{g})_2=V_0^{-1} \left(cB_2 (z\underline{m})_2 +cB_1v \right).$$
Then, from \eqref{7220}, we have
$$ \underline{m}_2 \vert z \vert^2 =V(z\underline{m})_2 + (1-c)v + cB_0 v+c^2 A_1 B_1vV_0^{-1}. $$
Substituting this into \eqref{7218} above, we obtain
\begin{equation} \label{7222}
\underline{m}_2 \vert z \vert^2 =V^2\underline{m}_2+(VcA_1+c^2 A_1 B_1)vV_0^{-1}+ (1-c)v+cB_0v.
\end{equation}
We see, then, for $c\leq1$
\begin{align}\label{Vlessthanz}
V=c^2A_1B_2V_0^{-1} + cB_1 < \vert z \vert.
\end{align}
This completes the proof of this proposition.
\end{proof}

\begin{remark}
	When $c>1$, this inequality \eqref{Vlessthanz} is not true for all small $z\in \mathbb{C}^+$, so it must be added that $c\leq 1$. Next, we show that the inequality \eqref{Vlessthanz} does not true when $ z $ is small.
	In fact, when $z$ tends to $0$,  we have 
	\begin{align}\label{remark2}
		\lim_{z\to 0}(-z g_n(z))=\lim_{z\to 0} \dfrac{1}{n} \sum_{j=N+1}^{n} v_j^* T_n v_j \geq 
		\lim_{z\to 0} \dfrac{1}{n} \sum_{j=N+1}^{n} \lambda_j (T_n),
	\end{align}
	where $ v_j, j=N+1,\dots, n $ are the eigenvectors corresponding to the $ n-N $ zero eigenvalues of $ B_n $ and $ \lambda_j(T_n), j=N+1,\dots, n$ are the $ n-N $ smallest eigenvalues of $ T_n $. By Cauchy-Schwarz inequality, we have
	\begin{align*}
		(1-1/c_n)^2 &= \left( \dfrac{1}{n} \sum_{j=N+1}^{n} \lambda_j^{1/2}(T_n) \lambda_j^{-1/2}(T_n)  \right)^2\\
		&\leq \left( \dfrac{1}{n} \sum_{j=N+1}^{n} \lambda_j(T_n) \dfrac{1}{n} \sum_{j=N+1}^{n} \lambda_j^{-1}(T_n)  \right)\\
		&\leq \left( \dfrac{1}{n} \sum_{j=N+1}^{n} \lambda_j(T_n) \dfrac{1}{n} \sum_{j=1}^{n} \lambda_j^{-1}(T_n)  \right).
	\end{align*}
By \eqref{remark2}, it is obvious that 
$$ \lim_{z\to0} (-zg_n(z))\le \frac1n\sum_{j=1}^n\lambda_j(T_n)\to \lambda_1.$$
Thus when $ n\to \infty $, once $(-zg_n(z))$ tends to a limit, then the limit should be falling in the interval $\big((1-1/c)^2/ \lambda_{-1}, \lambda_1)$.
Therefore, if $ g(z) = -1/(z (1 + cg(z)) \to \ug_0,$ then $\ug_0$ should be a positive constant. Then, by the second equation of \eqref{third}, $ \um(z) $ tends to $ \um_0=\um(0)$ which should be finite and positive. 
When considering the limit of $\um(z)$ and $\ug(z)$, we may allow $z\to 0$ through $\mathbb C^+$. 
Hence $(0,\um(0),\ug_0)$ is an extended solution to the equations (\ref{third}). This implies that $\ug_0$ and $\um(0)$ are unique. 
Consequently, $ V=\lim_{z\to 0}V(z)>0 $. So, the inequality \eqref{Vlessthanz} does not hold. 
\end{remark}

In order to obtain the properties of the LSD more conveniently, we give a preliminary result.
It is stated under conditions sufficient for this paper. 

\begin{lemma} \label{th77}
	We assume $c\leq1$. Under the assumptions on $H$  
	and the additional condition that $ \lambda_j \leq K $, $ j = -1, 1, $ for some constant $ K $, where $ \lambda_j = \left(\int t^j \mathrm{d}H(s, t)\right)^{1/2} $, we have
	\begin{description}
		
		\item[(a)] Under the condition that $ \int(s/t )\mathrm{d}H(s, t) \leq 1/c,$ for $ 0 < m \leq \vert z \vert \leq M < \infty, z\in\mathbb{C}^+ $, there exist $ 0 < \delta < \Delta < \infty  $, such that
		$$ \delta \le \vert \underline{m}(z)\vert \leq \Delta \mbox{ and }  \delta \le \vert \underline{g}(z)\vert.  $$
				
		\item[(b)] All the four quantities $ A_j $, $ B_j $, $ j=1,2 $ are bounded.
		
		\item[(c)] If $ \lim_{z_n\in\mathbb{C}^+}  z_n \to z $, $ \underline{m}(z_n) \to \tilde{m}(z) $ and $ \underline{g}(z_n) \to  \tilde{g}(z) $, then $ (z, \tilde{m}(z), \tilde{g}(z)) $ satisfies equations \eqref{third}.
		We call any triple $ (z, \tilde{m}(z), \tilde{g}(z)) $ formed in this way an ``extended solution".
		
		\item[(d)] For each $z\neq0$, the extended solutions are unique.
	\end{description}
\end{lemma}
\begin{remark} \label{re78.}
	The condition $ \int(s/t )\mathrm{d}H(s, t) \leq 1/c $ is one related to the ratio of signal to noise (RSN). We have difficulty proving the lower bound without this condition. 
	This condition means that the average RSN of each sample vector is not more than 1. 
	That is 
	\begin{align*}
	1 \geq \frac{1}{N} \sum_{k=1}^{N} T_n^{-\frac{1}{2}}(N^{-\frac{1}{2}} r_k)(N^{-\frac{1}{2}} r_k)^* T_n^{-\frac{1}{2}} = \frac{n}{N} \frac{1}{n} \mathrm{tr}(\frac{1}{N}R_n R_n^*) T_n^{-1}
	\to  c\int s/t  \mathrm{d}H(s, t).
	\end{align*}
	Thus, this condition seems reasonable. However, we hope the lower bound for $ m(z) $ can be proved without this additional condition.
\end{remark}
\begin{proof}[Proof of Lemma \ref{th77}]
	If (a) is not true for the upper bound, then there exists $ \underline{m}(z_n) \to \infty $. 
	Noting that, from \eqref{Vlessthanz}, $ cB_1 \leq \vert z \vert \leq M $. With the assumption $ \underline{m}(z_n) \to \infty $, it implies the right hand side of the first equation of \eqref{third} tends to zero while the left hand side does not.   
	The contradiction shows that  $ \underline{m}(z) $ is bounded from above. 
	Now, suppose there is a sequence $ \{z_n\} $ such that $ \underline{g}(z_n) \to 0 $. By the second equation of \eqref{third}, we have 
	$$ z_n \underline{g}(z_n)=-1+c\int \dfrac{t\underline{g}(z_n)}{1+s\underline{g}(z_n) +t\underline{m}(z_n)} \mathrm{d} H(s,t). $$ 
	As indicated by the second estimate in \eqref{7224}, the second term on the right side of the above and the left hand side both tend to zero, which results in a contradiction. 
	
	Finally, we show that $ \underline{m}(z) $ is bounded from zero. 
	Now, suppose there is a bounded sequence $ \{z_n\} $  such that 
	$ \underline{m}(z_n) \to 0 $. 
	From the Cauchy-Schwarz inequality, we find for some positive $M_1$,
	\begin{equation}
		\begin{aligned}
			\int \dfrac{\mathrm{d}H(s,t)}{\vert 1 + s\underline{g}(z_n) + t\underline{m}(z_n)\vert}\leq \lambda_{-1}\sqrt{B_1}\leq M_1,
			\\
			\int \dfrac{t\mathrm{d}H(s,t)}{\vert 1 + s\underline{g}(z_n) + t\underline{m}(z_n)\vert}\leq \lambda_{1}\sqrt{B_1}\leq M_1. \label{7224}
		\end{aligned}
	\end{equation}
	By the first equation of \eqref{third}, we have
	\begin{align}\label{7225}
	z_n \underline{m}(z_n) = -1 + c - c\int \dfrac{1}{1 + s \underline{g}(z_n) + t \underline{m}(z_n) }\mathrm{d}H(s, t).
	\end{align}
	Since it has been proven that $ \{z_n, \underline{g}(z_n)\} $ is bounded, there is a subsequence $ n' $ such that $ z_{n'} \to z_0  $ and 
	$ \underline{g}(z_{n'}) \to \underline{g}_0 $. 
	The first estimation in \eqref{7224} together with Fatou's Lemma shows that
	$$  \int \frac{1}{\vert 1 + s \underline{g}_0 \vert }\mathrm{d}H(s, t) \leq M_1.  $$
	Consequently, $ 1/(1 + s\underline{g}_0) $ is integrable with respect to $ H $, hence,
	\begin{align*}
	&\left\vert \int \frac{1}{1 + s \underline{g}(z_{n'}) + t\underline{m}(z_{n'}) }\mathrm{d}H(s, t) - \int \frac{1}{1 + s\underline{g}_0}\mathrm{d} H(s, t)    \right\vert  \\
	=&\left\vert \int  \frac{s(\underline{g}_0 - \underline{g}(z_{n'})) -t\underline{m}(z_{n'})}{(1 + s\underline{g}(z_{n'}) + t\underline{m}(z_{n'}))(1 + s\underline{g}_0)}\mathrm{d}H(s, t)     \right\vert \\
	\leq & \vert \underline{g}_0 - \underline{g}(z_{n'}) \vert \sqrt{A_1(n') A_1(0)} + \vert \underline{m}(z_{n'})\vert \vert z_{n'}\vert  \sqrt{B_1(n') B_1(0)},
	\end{align*}
	where $A_1(n')\ (B_1(n'))$ and $ A_1(0)\ (B_1(0)) $ are the $ A\ (B) $ functions defined by
	extended solutions $ (z_{n'}, g(z_{n'}), m(z_{n'})) $ and $ (z_0, g_0, 0) $, respectively. 
	
	In part (b), it will be proven that both the $ A $ and $ B $ functions are bounded, hence, the right hand side of the inequality above tends to zero. 
	Therefore, from \eqref{7225}, we
	obtain
	\begin{align}\label{7226}
	0 = -1 + c - c \int \frac{1}{1 + s\underline{g}_0} \mathrm{d}H(s, t) = -1 + c \int \frac{s\underline{g}_0}{1 + s\underline{g}_0}\mathrm{d} H(s, t) .
	\end{align}
	If the imaginary of $ \underline{g}_0 $ is positive, the equality above could not be true (including the case $ H(\{s = 0\}) = 1 $).  
	By Cauchy-Schwarz inequality, we have
	\begin{align*}
	&\int \frac{\vert s \underline{g}_0\vert}{\vert 1 + s \underline{g}_0\vert}\mathrm{d}H(s, t) \leq
	\left( \int(s/t)\mathrm{d}H(s, t) \int \frac{st \vert \underline{g}_0\vert^2}{\vert 1 + s\underline{g}_0 \vert^2}\mathrm{d}H(s, t)  \right)^{\frac{1}{2}}\\
	=& \sqrt{\int(s/t)\mathrm{d}H(s, t)\vert \underline{g}_0^2\vert A_2 }< \frac{1}{c}.
	\end{align*}
	Here the last inequality follows from the additional condition and \eqref{eqa0}.
	We reach a contradiction to \eqref{7226} and the proof of (a) is complete.
	
	Proof of (b). From \eqref{Vlessthanz}, we have seen that $ B_1\leq{\vert z\vert}/{c}\leq{M}/{c}$ by assumptions. 
	Also, we have from \eqref{eqa0} that $ A_2 \leq c^{-1}\vert \underline{g}\vert^{-2} \leq M_2 $ for some positive $M_2$, since $ \vert \underline{g}\vert \geq \delta > 0 $ proven in (a). Next, by \eqref{Vlessthanz}  we have for some  $ M_3 $,
	\begin{align}
	A_1B_2 \leq \vert z\vert \vert \underline{g}\vert^{-2}/c^2 < M_3. \label{eqjoint}  
	\end{align}
	 
	To show $A_1$ is bounded from above, we claim that there is a positive $t_0$ such that
	\begin{align*}
	\int\frac{I_{\{t>t_0\}}s\mathrm{d}H(s,t)}{\vert 1+s\underline{g}+t\underline{m}\vert^2}\ge 0.5A_1.
	\end{align*}
	If not, then for $t_0=0$,
	\begin{align*}
	\int\frac{I_{\{t>0\}}s\mathrm{d}H(s,t)}{\vert 1+s\underline{g}+t\underline{m}\vert^2}\le 0.5A_1,
	\end{align*}
	would imply either $A_1=0$ or $H({t=0})>0$, both of which contradict the assumption that $\lambda_{-1}\le K$.
	Consequently, we have
	\begin{align*}
	A_1&\le 2\int\frac{I_{\{t>t_0\}}s\mathrm{d}H(s,t)}{\vert 1+s\underline{g}+t\underline{m}\vert^2}\le \frac2{t_0}\int\frac{I_{\{t>t_0\}}st\mathrm{d}H(s,t)}{\vert 1+s\underline{g}+t\underline{m}\vert^2} \le  \frac2{t_0}A_2\le M.
	\end{align*}
	Then, we conclude that $ A_1 $ is bounded from above.
	
	Assuming that $B_2$ is not bounded for the bounded $z$, there must be a $z_0\in \mathbb C^+$such that $B_2(z_0)=\infty$.
	According to (\ref{eqjoint}), it results in $A_1(z_0)=0$.
	This implies that the marginal for $ s $ of $ H $ is  concentrated at zero. Therefore, 
	\begin{eqnarray*}
		B_2=\int\frac{t^2\mathrm{d}H(s,t)}{\vert 1+t\underline{m}(z_0)\vert^2}=\infty.
	\end{eqnarray*}
	This equation implies that $\underline{m}(z_0)$ is real and not positive. By part (a), $\underline{m}(z_0)\le -\delta$ for some $\delta>0$. If $t>2/\delta$, then $\vert 1+t\underline{m}(z_0)\vert >t\delta -1>0.5t\delta$,  therefore
	\begin{eqnarray*}
		\int\frac{I_{\{t\le 2/\delta\}}t^2\mathrm{d}H(s,t)}{\vert 1+t\underline{m}(z_0)\vert^2}=\infty.
	\end{eqnarray*}
	This is impossible because
	\begin{eqnarray*}
		&&\int\frac{I_{\{t\le 2/\delta\}}t^2\mathrm{d}H(s,t)}{\vert 1+t\underline{m}(z_0)\vert^2}
		\le \frac2{\delta}\int\frac{I_{\{t\le 2/\delta\}}t\mathrm{d}H(s,t)}{\vert 1+t\underline{m}(z_0)\vert^2}\le \frac2{\delta}B_1\le M.
	\end{eqnarray*}
	Thus the contradiction proves that $ B_2 $ is bounded from above.
	
	For (c), by applying Fatou's Lemma when the limiting $z$ is real, the same bounds on $A_j$, $B_j$ hold true for extended solutions to equation \eqref{third}.  Also, from \eqref{7224}, using Fatou's Lemma, the integrals on the right side of \eqref{third} also exist and are bounded.
	We have
	\begin{align*}
	\begin{split}
	&\int\dfrac{\mathrm{d}H(s,t)}{1+s\underline{g}(z_n)+t\underline{m}(z_n)}-\int\dfrac{\mathrm{d}H(s,t)}{1+s\tilde{g}
		+t\tilde{m}}\\
	=&(\tilde{g}-\underline{g}(z_n)) \int \dfrac{s\mathrm{d}H(s,t)}{(1 + s\underline{g}(z_n) +t\underline{m}(z_n))(1 + s\tilde{g} + t\tilde{m})} \\&
	+(\tilde{m}- \underline{m}(z_n)) \int \dfrac{t\mathrm{d}H(s,t)}{(1 + s\underline{g}(z_n) + t\underline{m}(z_n))(1 + s\tilde{g} + t\tilde{m})} \to 0,\\
	&\int\dfrac{t\mathrm{d}H(s,t)}{1+s\underline{g}(z_n)+t\underline{s}(z_n)}-\int\dfrac{tdH(s,t)}{1+s\tilde{g}
		+t\tilde{m}}\\
	=&(\tilde{g}-\underline{g}(z_n)) \int \dfrac{st\mathrm{d}H(s,t)}{(1 + s\underline{g}(z_n)+t\underline{m}(z_n))(1 + u\tilde{g} + t\tilde{m})} \\&
	+(\tilde{m}- \underline{m}(z_n)) \int \dfrac{t^2\mathrm{d}H(s,t)}{(1 + s\underline{g}(z_n) + t\underline{m}(z_n))(1 + u\tilde{g} + t\tilde{m})} \to 0.
	\end{split}
	\end{align*}
	With $A_j(n),B_j(n)$, $A_j,B_j$ having their obvious meanings, for some positive $M$, we have
	\begin{align*}
	&\int \dfrac{s\mathrm{d}H(s,t)}{\vert (1 + s\underline{g}(z_n) + t\underline{m}(z_n))(1 + s\tilde{g} + t\tilde{m})\vert}
	\leq \sqrt{A_1(n)A_1} \leq M,\\
	&\int \dfrac{t\mathrm{d}H(s,t)}{\vert (1 + s\underline{g}(z_n) + t\underline{m}(z_n))(1 + s\tilde{g} + t\tilde{m})\vert}
	\leq \sqrt{B_1(n)B_1} \leq M,\\
	&\int \dfrac{st\mathrm{d}H(s,t)}{\vert (1 + s\underline{g}(z_n) + t\underline{m}(z_n))(1 + s\tilde{g} + t\tilde{m})\vert}
	\leq \sqrt{A_2(n)A_2 }\leq M,\\
	&\int \dfrac{t^2\mathrm{d}H(u,t)}{\vert(1 + s\underline{g}(z_n) + t\underline{m}(z_n))(1 + s\tilde{g} + t\tilde{m})\vert}
	\leq \sqrt{B_2(n)B_2} \leq M.
	\end{align*}

    Proof of (d). Take note that in (b), both $A_1$ and $B_2$ are bounded from below. Thus, based on \eqref{Vlessthanz}, $V_0^{-1}$ is also bounded. Therefore, through the application of Fatou's Lemma and considering that $\underline g$ is bounded from below, we have from \eqref{Vlessthanz} \begin{equation}
    	V=c^2A_1B_2V_0^{-1} + cB_1 \leq \vert z \vert \label{Vlessthanorequalz},
    \end{equation}
    which holds for all $z$ with $\Im z\ge0$.
	
	Suppose for $z\ne 0$ with $\Im z\ge0 $, there are two sets of extended solutions 
	$ (\underline{m}_{(j)}, \underline{g}_{(j)}) $, $ j = 1, 2. $
	Multiplying $ \underline{m}_{(j)} $ on both sides of the first equation of \eqref{third} and taking the difference on both sides, we obtain
	\begin{align}
	\begin{split}
	&z(\underline{m}_{(1)}-\underline{m}_{(2)})=c\tilde{A}_1(\underline{g}_{(1)}-\underline{g}_{(2)})
	+c\tilde{B}_1(\underline{m}_{(1)}-\underline{m}_{(2)}),\\ 
	&0=\dfrac{\underline{g}_{(1)}-\underline{g}_{(2)}}{\underline{g}_{(1)}\underline{g}_{(2)}}
	-c\tilde{A}_2(\underline{g}_{(1)}-\underline{g}_{(2)})-c\tilde{B}_2(\underline{m}_{(1)}-\underline{m}_{(2)}),\label{7227}
	\end{split}
	\end{align}
	where 
	$$\tilde{A}_j=\int \dfrac{st^{j-1} \mathrm{d}H(s,t)}{(1+s\underline{g}_{(1)} +t\underline{m}_{(1)})(1+s\underline{g}_{(2)}+t\underline{m}_{(2)})}, \quad j=1,2$$
	$$\tilde{B}_j=\int\dfrac{t^{j}\mathrm{d}H(s,t)}{(1+s\underline{g}_{(1)}+t\underline{m}_{(1)})(1+s\underline{g}_{(2)}+t\underline{m}_{(2)})},\quad  j=0,1,2.$$
	
	Let $ A^{(i)}_j  $, $ B^{(i)}_j  $, $ i, j = 1, 2 $ be the $ A_j $, $ B_j $ functions defined in \eqref{7216} by the $ i $-th extended solutions.  
	Then by Cauchy-Schwarz inequality, we obtain 
	$$\vert c\tilde A_2\vert \leq c(A_2^{(1)})^{1/2}(A_2^{(2)})^{1/2}<\frac1{\vert \underline{g}_{(2)}\vert \vert  \underline{g}_{(1)}\vert},$$
	yielding
	$$\left\vert\frac1{\underline{g}_{(2)} \underline{g}_{(1)}}-c\tilde A_2\right\vert \ge\left\vert \frac1{\vert \underline{g}_{(2)}\vert \vert \underline{g}_{(1)}\vert}-\vert c\tilde A_2\vert \right\vert \ge\frac1{\vert \underline{g}_{(2)}\vert \vert \underline{g}_{(1)}\vert}-c\sqrt{A_2^{(1)}A_2^{(2)}}
	>0.$$
	Consequently, the second equation in \eqref{7227} implies that $\um_{(1)}=\um_{(2)}$ if and only if $\ug_{(1)}=\ug_{(2)}$.
	Suppose $ \underline{m}_{(1)}\ne\underline{m}_{(2)}$, we can then write
	$$\frac{\underline{g}_{(1)}-\underline{g}_{(2)}}{\underline{m}_{(1)}-\underline{m}_{(2)}}=
	\frac{c\tilde{B}_2}{\frac1{\underline{g}_{(2)} \underline{g}_{(1)}}-c\tilde A_2}.$$
	From \eqref{7227} and \eqref{Vlessthanorequalz}, applying Cauchy-Schwarz inequality, and using the inequality
	$$(ab-cd)^2\ge(a^2-c^2)(b^2-d^2),\text{ for real }a,b,c,d,$$
	we derive
	\begin{align}
	\begin{split}
	\vert z\vert=&\left\vert\frac{c\tilde{B}_2c\tilde{A}_1}{\left( \frac{1}{\underline{g}_{(2)} \underline{g}_{(1)}} \right)-c\tilde{A}_2 }+c\tilde{B}_1 \right\vert 
	\leq  \dfrac{c^2\sqrt{B^{(1)}_2 B^{(2)}_2 A^{(1)}_1 A^{(2)}_1}}{\left( \frac{1}{\vert \underline{g}_{(2)}\vert \vert \underline{g}_{(1)}\vert}  \right)-c\sqrt{A^{(1)}_2 A^{(2)}_2} } 
	+c\sqrt{B^{(1)}_1 B^{(2)}_1}\\ 
	\leq &\left(    \left( \dfrac{c^2 B^{(1)}_2 A^{(1)}_1}{\left( \frac{1}{\vert \underline{g}_{(1)}\vert^2} -cA^{(1)}_2 \right) } +c B^{(1)}_1 \right)    
	\left( \dfrac{c^2 B^{(2)}_2 A^{(2)}_1}{\left( \frac{1}{\vert\underline{g}_{(2)}\vert^2} -cA^{(2)}_2 \right) } +c B^{(2)}_1 \right)             \right)^{\frac{1}{2}} \\
	\leq& \vert z\vert.\label{7228}
	\end{split}
	\end{align}
	It follows that when $v>0$, \eqref{Vlessthanz} is strict and thus we get an immediate contradiction.  
		
	Now, consider the case that $v=0$ and $z=x\in\mathbb R$, $x\neq0$.   
	For the case, $\Im \um_{(1)}=0$ or $\Im \um_{(2)}=0 $, dividing by $v$ both side of \eqref{7222} and then making $v$ tends to $0$, we get that \eqref{Vlessthanz} holds strictly, because 
	\begin{align*}
		\lim_{v\to 0} \frac{\um_2(z)}{v} 
		&=\lim_{v\to 0} \dfrac{i\Im(\um(x+iv)-\Im(\um(x)))}{x+iv-x}\\
		&=\lim_{v\to 0} \Im i\dfrac{\um(x+iv)-\um(x)}{x+iv-x} \to \Im(i\um'(x))=\um'(x) >0.
	\end{align*}
	This leads to a contradiction also in \eqref{7228}.
	
	Then we consider the case that $v=0$ and $\Im \um_{(j)}>0$.      
	In order to have the last inequality in \eqref{7228} to be an equality, we need to have the factors in \eqref{7228} are both $\vert z \vert $ and the two vectors
	$$\left(\sqrt{\frac{c^2A_1^{(1)}B_2^{(1)}}{\frac1{\vert \underline g_{(1)}\vert^2} -cA_2^{(1)}}}, \sqrt{cB^{(1)}_1}\right), \left(\sqrt{\frac{c^2A_1^{(2)}B_2^{(2)}}{\frac1{\vert\underline g_{(2)}\vert^2} -cA_2^{(2)}}},\sqrt{cB^{(2)}_1}\right)$$ 
	are proportional to each other, say the first vector is $k$ times the second vector, necessarily $k$ is positive, so we have 
	$$\vert z \vert =\left( \dfrac{c^2 B^{(1)}_2 A^{(1)}_1}{\left( \frac{1}{\vert \underline{g}_{(1)}\vert^2} -cA^{(1)}_2 \right) } +c B^{(1)}_1 \right)=k^2\left( \dfrac{c^2 B^{(2)}_2 A^{(2)}_1}{\left( \frac{1}{\vert\underline{g}_{(2)}\vert^2} -cA^{(2)}_2 \right) } +c B^{(2)}_1 \right)=k^2\vert z \vert.$$ 
	Therefore, $k$ must be 1.  
	Hence, we have $B_1^{(1)}=B_2^{(2)}$. 
	In the last step of application of Cauchy-Schwarz inequality to get $\vert \bbE(X Y)\vert = \sqrt{\bbE(\vert X\vert^2)\bbE(\vert Y\vert^2)}= \bbE(\vert X\vert^2)=\bbE(\vert Y\vert^2)$, one needs  $ X=\overline Y$ which is impossible because the imaginary parts of both $X$ and $Y$ are positive. 
	So we obtain at least one contradiction to \eqref{7228}. 
	Thus, we finally proved that it is impossible to have  $\underline m_{(1)}\ne \underline m_{(2)}.$ 
	This completes the proof of the theorem.
\end{proof}

Dozier and Silverstein \cite{dozier2007analysis} extended the analysis for LSD of generalized M-P law to the LSD of the non-central sample covariance matrix.
Theorem 2.1 in \cite{dozier2007analysis} can be extended to the solution for model \eqref{Bnul}. We shall prove the following theorem under the assumption $c\leq1$ and the conditions imposed on the limiting $H$ in Theorem \ref{th77}.  

\begin{theorem}\label{th713} 
	Assume $c\leq1$ and the conditions imposed in the limiting $H$ in Theorem \ref{th77}.  
	Suppose $\um(x) $ is the solution to \eqref{third} for $x\neq0$. Then 
	$ x \in \mathbb{R}\backslash\{0\} $, $ \lim_{z\in \mathbb{C}^+ \to x}m_{\underline F} (z) \equiv \um(x) $ exists. 
	The function $\um$ is continuous on $ \mathbb{R}\backslash\{0\} $, and $ F $ has a continuous derivative $ f $ on 
	$ \mathbb{R}\backslash\{0\} $ given by $ f(x) = \frac{1}{\pi} \Im \um(x) $. Furthermore, if $ \Im 
	\um(x) > 0$  $ (f(x) > 0)  $ for $ x \in \mathbb{R}^+ $, then 
	the density $ f $ is analytic about $ x $.
\end{theorem}
\begin{proof}[Proof of Theorem \ref{th713}]
	The first conclusion of Theorem \ref{th713} is a special case of Lemma \ref{th77}. Because of (a), (c), and (d), as $z\in\mathbb{C}^+\to x\neq0$, $\um(z)\to \um(x)$, solution to \eqref{third}.  Thus, the Stieltjes transform extends uniquely on $\mathbb{R}/\{0\}$.
	Then, due to Theorems 2.1 and 2.2 of \cite{silverstein1995analysis}, we have that for all $x\neq0$, $\underline F $ is continuously differentiable and the density is given by 
	$$ f(x) = \frac{1}{\pi}\Im(\um(x))=\frac{1}{\pi}\lim_{v \downarrow 0} \Im(\um(x + iv)). $$
	Next, we show that the density $ f(x) $ is analytic when $ \Im(\um(x)) > 0 $. As a result, the denominator of the fractions in \eqref{third} cannot be zero at the point $ (m(x), g(x), x) $. 
	Hence, the two integrals in \eqref{third} are analytic in some neighborhood $ D_x $ of $ (\um(x), \ug(x), x) $. 
	We shall employ the implicit function theorem to show that $ \um(x) $ is analytic in some neighborhood $ D_x $.
	Rewrite the two equations in \eqref{third} as 
	$ G_{\um}(\um, \ug, z) = 0 $ and $ G_{\ug}(\um, \ug, z) =0 $, where
	\begin{align}\label{7329}
		&G_{\um} = z\um + (1 - c) + c \int \dfrac{\mathrm{d}H(s, t)}{1 + s\ug + t\um},\\ \nonumber
		&G_{\ug} = z + \frac{1}{\ug} - c \int \dfrac{t\mathrm{d}H(s, t)}{1 + s\ug + t\um}.
		\end{align}
	Then
	\begin{align*}
		&\frac{\partial}{\partial \um} G_{\um}(\um, \ug, x) = x - \hat{B}_1 c,\\
		&\frac{\partial}{\partial \ug} G_{\um}(\um, \ug, x) = - \hat{A}_1 c, \\
		&\frac{\partial}{\partial \ug} G_{\ug}(\um, \ug, x) = - \frac{1}{\ug^2} + \hat{A}_2 c, \\
		&\frac{\partial}{\partial \um} G_{\ug}(\um, \ug, x) = \hat{B}_2 c,
	\end{align*}
	where
	\begin{align*}
		&\hat{A}_1=\int \dfrac{s\mathrm{d}H(s,t)}{(1 + s\ug + t\um)^2},\\
		&\hat{B}_1=\int \dfrac{t\mathrm{d}H(s,t)}{(1 + s\ug + t\um)^2}, \\
		&\hat{A}_2=\int \dfrac{st\mathrm{d}H(s,t)}{(1 + s\ug + t\um)^2}, \\
		&\hat{B}_2=\int \dfrac{t^2\mathrm{d}H(s,t)}{(1 + s\ug + t\um)^2}.
	\end{align*}
	By the implicit function theorem, $ \um $ and $ \ug $ are analytic in some neighborhood $ D_x $ of $ x $ if the absolute value of the below determinant is	greater than zero.  
	We have for $x>0$,
\begin{align*}
&\left\| \begin{array}{cc} 
\frac{\partial}{\partial \um} G_{\um}(\um, \ug, x) & \frac{\partial}{\partial \ug} G_{\um}(\um,\ug, x) \\
\frac{\partial}{\partial \um} G_{\ug}(\um, \ug, x) & \frac{\partial}{\partial \ug} G_{\ug}(\um, \ug, x)
\end{array} \right\| \\
> &\vert (x-\hat{B}_1 c)(\frac{1}{\ug^2}-\hat{A}_2 c)-c^2\hat{A}_1 \hat{B}_2\vert
\geq (x - B_1 c)( \frac{1}{\vert \ug \vert^2} - A_2 c) - c^2 A_1 B_2 \\
= &( \frac{1}{\vert g \vert^2} - A_2 c)(\vert x \vert - c B_1 - c^2 A_1 B_2 ( \frac{1}{\vert \ug \vert^2} - A_2 c)^{-1}) = 0.
\end{align*}
From the same argument given above, the ``$>$'' comes from the fact that the absolute values of the $\hat A_i$, $\hat B_i$ are strictly less than their real counterparts, and ``$=$''  is due to \eqref{7228}.
The proof is complete.
\end{proof}

\section{The Support of  $ F $}\label{sec3}

In this section, we present results on the support of the limiting distribution $ F $. 
The support of a distribution function $ F $ is the set of all points $ x $ satisfying 
$ F(x + \varepsilon) - F(x - \varepsilon) > 0  $ for all $ \varepsilon> 0. $
Let $ S_F $ and $ S_H $ denote the support of $ F $ and $ H $, respectively. 
Clearly, by the definition of $F$ and $H$, we have $ S_F \subset [0, \infty)  $ and  $ S_H \subset [0, \infty)\times (0, \infty) $. 

We will concentrate our investigation of the support, $S_F$ of $F$ by the support, $S_{\uF}$ of $\uF$. The latter can be found by determining those values of the real line which are in its complement,  $S_{\uF}^c$.
Corresponding to \cite{dozier2007analysis}, we begin our analysis of the support of the LSD $ F $ with the following result.
\begin{theorem}\label{th714}
	When $ c \leq 1 $, and assuming the conditions imposed on H in Theorem \ref{th77}, the LSD $ F $ determined by \eqref{second} has no mass at zero.
\end{theorem}
\begin{proof}[Proof of Theorem \ref{th714}]
	Using the formula $ G(\{0\}) = \lim_{v\downarrow 0}(-iv m_G(iv))$ for any probability measure $ G $, and by the relation between $m(z)$, $ \um(z) $ and \eqref{third}, we have
	\begin{align*}
	F(\{0\}) &= \lim_{v \downarrow 0}(-iv m(iv))\\
	&= \lim_{v \downarrow 0} \int \dfrac{\mathrm{d}H(s, t)}{1 + s\underline{g}(iv) + t\underline{m}(iv)} \\
	&= \lim_{v \downarrow 0} \int \dfrac{-iv\mathrm{d}H(s, t)}{-iv - siv\underline{g}(iv) - tiv\underline{m}(iv)}\\
	&= \lim_{v \downarrow 0} \int \dfrac{-iv\mathrm{d}H(s, t)}{-siv\underline{g}(iv) + t\underline{F}(\{0\})},
	\end{align*}
	where $ \underline{F}(\{0\}) = F(\{0\})\geq 0 $ if $c=1 $ and $ \underline{F}(\{0\}) \geq (1 - c) $ if $ c<1 $.   
	Therefore $ \lim(-iv\underline{g}(iv)) \geq 0 $.  
	Thus $ F(\{0\}) > 0 $ will lead to a contradiction because the left hand side would be positive while the right hand would necessarily be zero.	
\end{proof}

To identify points, $x$ in $S_{\uF}^c$, we only need to consider $x\ge0$. When $c\leq1$, according to the preceding theorem, $x>0$. Because if $(0,\epsilon)\subset S_{\uF}^c$, it necessarily implies that $\uF$ has a mass of $1-c$ at 0.

We next claim that the supports of $\uF$ and $P_{\ug}$ are identical, and $P_{\ug}$ is a probability measure whose the Stieltjes transform is $\ug(z)$. Notice that $x\in\mathbb R$ outside the support of a measure implies the Stieltjes transform of the measure is real at $x$. If $(x,\um,\ug)$, with $x$ real, is an extended solution, and if  $\um$ is real, then from the first equation in \eqref{third}, $\ug$ must also be real. If $\ug$ is real, then the second equation in \eqref{third} requires $\um$ to be real. This proves the claim. It will be useful to consider $S_{P_{\ug}}^c$ in the arguments below.

The following presents a scheme for computing solutions $z,\um,\ug$ of \eqref{third}. 
Equating the two equations in \eqref{third}, we get 
$$ -\frac{1-c}{\underline{m}} - \frac{c}{\underline{m}}\int \dfrac{\mathrm{d}H(s, t)}{1 + s\underline{g} + t\underline{m}} = -\frac{1}{\underline{g}} + c \int \dfrac{t\mathrm{d}H(s, t)}{1+s\underline{g}+t\underline{m}}.  $$
This is equivalent to
\begin{equation}\label{justumug} 
c\ug^2\int \dfrac{s \mathrm{d}H(s, t)}{1 + s\underline{g} + t\underline{m}} +{\underline{m}}-{\underline{g}}=0.  
\end{equation}
When  $(z,\um,\underline{g})$ is an extended solution to  \eqref{third}, we claim there is only one $\um$ for every $\ug$ solving this equation. 
If not, suppose there are two different $ \underline{m}_1 $ and $ \underline{m}_2 $ satisfying the equation. Taking the difference between the two equations
we obtain,
\begin{eqnarray*}
	c(\underline{m}_2-\underline{m}_1)\underline{g}^2\int \dfrac{st \mathrm{d}H(s, t)}{(1 + s\underline{g} + t\underline{m}_1)(1 + s\underline{g} + t\underline{m}_2)} +\underline{m}_1-\underline{m}_2=0.
\end{eqnarray*} 
Consequently, we have 
\begin{eqnarray*}
	c\int \dfrac{st \mathrm{d}H(s, t)}{(1 + s\underline{g} + t\underline{m}_1)(1 + s\underline{g} + t\underline{m}_2)} =\frac{1}{\underline{g}^2}.
\end{eqnarray*} 
Therefore, by Cauchy-Schwarz inequality
\begin{eqnarray*}
	\frac{1}{\vert \underline{g}\vert^2}\le  \left(c\int \dfrac{st \mathrm{d}H(s, t)}{\vert 1 + s\underline{g} + t\underline{m}_1\vert^2}\cdot c\int \dfrac{st \mathrm{d}H(s, t)}{\vert 1 + s\underline{g} + t\underline{m}_2\vert^2}\right)^{1/2}<\frac{1}{\vert \underline{g}\vert^2}.
\end{eqnarray*}
Here, the last inequality follows from (\ref{eqa0}) (as mentioned above, true for real $x$). The contradiction proves our assertion.   

\begin{theorem}\label{th716}
	Assume $c\leq1$ and the conditions imposed on $H$ in Theorem \ref{th77} hold.  Let $x_0 \in S^c_{\underline{F}}\cap\mathbb R^+$,
	\begin{description}
		\item[(a)] Then $ \underline{m}(z)=\int (t-z)^{-1}\mathrm{d}\underline{F}(t) $ is analytic in a neighborhood $ D_{x_0} $ of $ x_0 $ and there exists a co-solution $ \underline{g}(z) $ which is also analytic in $ D_{x_0} $. 
		The triple $ (x, \underline{m}(x), \underline{g}(x)) $, $x\in D_{x_0}\cap\mathbb{R}^+$ is an extended solution to  
		\eqref{third} with $V < x$.
		\item[(b)] For any support point $ (s, t) $ of $ H $, $ s\underline{g}(x_0) + t\underline{m}(x_0)  \neq -1 $.
	\end{description}
	On the other hand, if $x_0,\um_0,\ug_0$, with $x_0>0$,  form a real extended solution to \eqref{third} satisfying (b), then from \eqref{justumug}, 
	there exists a real analytic function $x=x(\ug)$, defined in an interval containing $\ug_0$ which satisfy \eqref{third}, and if $x'(\ug_0) \neq0$, then $x_0\in S_{\uF}^c$.
\end{theorem} 
\begin{proof}[Proof of Theorem \ref{th716}]
	When $x_0$ is outside the support of a probability distribution, its Stieltjes transform exists at $x_0$, is analytic in a real neighborhood of $x_0$, and 
	has a positive derivative at each point in this neighborhood.   So for $x_0 \in S^c_{\underline{F}}\cap\mathbb R^+$,  there is a constant $ \varepsilon \in  (0, x_0) $ such that $ U_{x_0} = (x_0-\varepsilon, x_0 + \varepsilon) \subset S^c_{\underline{F}}  $. Let $ D_{x_0} $ denote the ball in the complex plane centered at $ x_0 $ with radius $ \varepsilon $. 
	Then, for any $ z\in D_{x_0} $, $ \tilde{\um}(z) = \int (t - z)^{-1}\mathrm{d}\underline{F}(t) $ is well defined and analytic.
	
	Then, $ \underline{m}(x + iv) =\int (t - (x + iv))^{-1}\mathrm{d}\underline{F}(t) $ and
	$ \tilde{\um}(z) $ are identical.
	Letting $ v \to 0 $, i.e. $ z = x + iv \xrightarrow{\mathbb{C}^+} x $, and applying Theorem \ref{th713}, we find that $ \underline{m}(x) $, $ \underline{g}(x) $ exist and satisfy \eqref{third}.
	By the unique extension theorem of analytic functions, $ \underline{m}(x) = \tilde{\um}(x) $ for all $ x \in U_{x_0} $. In the following, we will identify $ \tilde{\um} $ with $ \underline{m} $.
	
	The same argument applies for $P_{\ug}$, so $\ug(z)$ is also analytic in a neighborhood of $x_0$.  Let $D_0$ and $U_0$ be the smaller of
	the two respective sets associated with $\um$ and $\ug$.
	Suppose there is a support point $ (s, t) $ of $ H(s, t) $ such that 
	$ s\underline{g}(x_0)+t\underline{m}(x_0)+1 = 0 $. 
	Taking the derivative of the first equation of \eqref{third}, we get
	$$ \frac{\mathrm{d}(zm)}{\mathrm{d}z}=c\int \dfrac{s\ug'+t\um'}{(1+ s\ug + t\um)^2}\mathrm{d}H(s,t). $$ 
	Lemma \ref{le717}, which will be given later, suggests there is a point $ x^* $ arbitrarily close to $ x_0 $ such that
	$ \int (1 + s\underline{g}(x^*) + t\underline{m}(x^*))^{-2} \mathrm{d}H(s, t) = \infty $. This contradicts $ \underline{m}(x) $ being analytic in a neighborhood of $x_0$. Thus, condition (b) holds.
	
	Since $ \underline{m}(x) $ for $ x\in U_{x_0} $ is real with $ \underline{m}'(x)>0 $, we can apply the inverse function theorem that there exists a unique analytic function $ z(\underline{m}) $ defined on some neighborhood $ D_{\um_0}$ in the $ \underline{m} $ plane. 
	This function satisfies $ x = z(\underline{m}(x)) $, where $x\in U_{x_0}$. 
	Furthermore, as $ z'(\underline{m}(x)) = 1/\underline{m}'(x) $, we can deduce that 
	$ z'(\underline{m}(x)) > 0 $.
	We can compute the derivative of $ \underline{m}(x) $ as follows. Recalling the definitions of $G_{\um}$ and $G_{\ug}$, we have 
	$$ \frac{\partial}{\partial z}G_{\um}+{\um}'\frac{\partial}{\partial \um}G_{\um}+ {\ug}'\frac{\partial}{\partial{\ug}} G_{\um} = \um + \um'x - \um'c B_1 - c\ug' A_1 = 0, $$
	$$ \frac{\partial}{\partial z}G_{\ug}+\um'\frac{\partial}{\partial{\um}}G_{\ug}+ {\ug}'\frac{\partial}{\partial{\ug}} G_{\ug} = 1 +  \um'c B_2 -\frac{\ug'}{{\ug}^2}+ c {\ug}' A_2 = 0.  $$
	From the second equation, we get $$  \ug' = (1 + {\um}'c B_2)/({\ug}^{-2} - c A_2),  $$
	and then substituting it to the first equation
	\begin{align}\label{7330}
	\begin{split}
	{\um}'= [(-{\um}+c A_1)/({\ug}^{-2} - c A_2)](x - V )^{-1}.
	\end{split}
	\end{align}
	Therefore
	\begin{align}\label{zprime}
	\begin{split}
	z'=z'(\um)= (x- V)[(-\um +c A_1)/({\ug}^{-2} - cA_2)]^{-1}.
	\end{split}
	\end{align}
	
	Conversely, suppose $x_0$, $\um_0$ and $\ug_0$ with $x_0\neq0$, form a real extended solution to \eqref{third} satisfying (b).   
	Since $ \inf\{\vert s\ug_0 + t\um_0 + 1\vert : (s, t) \in S_H\} > 0 $ (since $ S_H $ is a closed set), the integrals in \eqref{third} as well as the integral in \eqref{justumug} are analytic functions of $ \um $ and $\ug$ in a neighborhood of $(\um_0,\ug_0)$.
	Notice that the partial derivative of \eqref{justumug} with respect to $\um$ is $$-c\ug^2A_2+1=\ug^2V_0>0.$$ 
	Therefore, by the implicit function theorem, $\um=\um(\ug)$ is uniquely defined and is analytic in a neighborhood of $\ug_0$ and with either of the two equations in \eqref{third}, we determine $z$.   
	So we see that in a neighborhood of $\ug_0$, $z=z(\ug)$ is analytic in a neighborhood of $\ug_0$.
	which are solutions to \eqref{third}.  
	When $z\in\mathbb{C}^+$, because of uniqueness we have $m_{\uF}(z)=\um(\ug)\in\mathbb{C}^+$ and $m_{\ug}(z)=\ug\in\mathbb{C}^+$.  
	We also see that \eqref{justumug} yields real $\um$ for each real $\ug$, and so $x=x(\ug)$ is a real-valued function for $\ug$ lying in an interval containing $\ug_0$.
	
	If $x'(\ug_0)\neq0$, the inverse function theorem ensures the existence of an analytic function $\ug=\ug(z)$ for $z\in\mathbb C$ within a neighborhood of $x_0$. This function serves as the inverse of $z(\ug)$, such that $z$, $\um(\ug)$, and $\ug$ provide a solution to \eqref{third}.
	Since $(x_0,\um_0,\ug_0)$ is an extended solution, there exists a sequence $z_n\in\mathbb{C}^+\to x_0$ and $m_{\uF}(z_n),m_{\ug}(z_n)\to \um_0,\ug_0$.   
	Therefore $\ug(z)$ is the analytic extension of $m_{\ug}$ onto a real interval that includes $x_0$. And $\ug(z)$is real-valued when $z$ is real, which implies that the density of $P_{\ug}$ exists and is zero in an interval containing $x_0$, so that necessarily $x_0$ is outside the support of $P_{\ug}$ and $\uF$.  
	Since $x_0$ is outside the support of $P_{\ug}$, $m_{\ug}(z_0)$ and $x'(\ug_0)$ are necessarily positive.
\end{proof}

\begin{lemma}\label{le717}
	Suppose $ H(s, t) $ is a measure supported by a closed subset of the first quadrant. 
	If there is a support point $ (s_0, t_0) \ne(0,0)$ such that $ g_0 s_0 + m_0 t_0 + c = 0, $
	where $c>0$ is a constant,
	$ g_0 = g(w_0) $ and $ m_0 = m(w_0) $ are functions defined on the interval 
	$ U = (w_0 -\eta, w_0 + \eta) $ whose derivative are not less than $ k \in (0, m) $. 
	Then there exists a $ w^*\in U $ such that 
	$ \int \vert sg(w^*) + tm(w^*) + c\vert^{-2}\mathrm{d}H(s, t) = \infty $.
\end{lemma}
\begin{proof}[Proof of \ref{le717}]
	By assumptions, for any support point $ (s,t) $  such that $ s+t > \phi > 0 $, where $\phi$ is the distance of the support of $H$ to the origin, we have
	\begin{equation}\label{7332}
	\int_{0}^{\eta} (sg'(w) + tm'(w))\mathrm{d}w \geq k \phi \eta > 0,
	\end{equation}
	and similiarly
	$$ \int_{0}^{-\eta} (sg'(w) + tm'(w))\mathrm{d}w \leq - k \phi \eta < 0. $$
	Without loss of generality, we may assume that $ c = 1 $. 
	Since $ (s_0, t_0) $ is a support point of $ H $, which is not the origin, let the constant $ d > 0 $ be such that
	$ d < k \phi \eta / \sqrt{2(g^2_0 + m^2_0) }$, and define $ R_0 $ to be a square contained in the first quadrant, containing the point $ (s_0, t_0) $, the edge length $ d > 0 $ and does not cover the origin. There is a positive constant $ \alpha > 0 $ such that $ H(R_0) = \alpha $.
	
	\begin{remark}\label{re718}
		If $ (s_0, t_0) $ is on one coordinate axis, we can take $ (s_0, t_0) $ as an inner point of one edge of $ R_0 $. If $ (s_0, t_0) $ is an inner point of the first quadrant, then we can take 
		$ (s_0, t_0) $ as an inner point of $R_0$. In either way, we may guarantee $ H(R_0) > 0 $.
	\end{remark}
	Split the square into four squares by equally dividing each edge into two. Denote one with $ H $ measure no less than $ \alpha/4 $ by $ R_1 $. 
	Inductively, split $ R_n $ similarly into four small squares and denote the one with $ H $ measure no less than $ \alpha4^{-n-1} $ by $ R_{n+1} $. 
	Note that for any $ n $, $ H(R_n) \geq 4^{-n}\alpha $. Write the center of $ R_n $ as $ (s_n, t_n) $ for all $ n \geq 1 $. By the nested interval theorem, $ (s_n, t_n) $ tends to a limit 
	$ (s^*, t^*)\in R_n, $ for all $ n $. Also, by the construction of the squares, we know that 
	$\|(s_{n+1}- s_n, t_{n+1} - t_n)\| =\sqrt2 2^{-n-1} d. $
	
	Write $ Q_n(w) = s_n g(w) + t_n m(w) + 1$,
	\begin{align*}
	Q_n(w) =& Q_n(w_0) +\int_{w_0}^{w} Q_n'(w) \mathrm{d}w \\
	=&((s_n - s_0)g_0 + (t_n - t_0) m_0 + \int_{w_0}^{w} Q_n'(w) \mathrm{d}w)\\
	&\left\{
	\begin{array}{rcl}
	&\geq& -d \sqrt{2(g^2_0 + m^2_0) } + k \phi \eta >0  \mbox{ when } w=w_0+\eta, \\
	&\leq& d \sqrt{2(g^2_0 + m^2_0) } - k \phi \eta <0  \mbox{ when } w=w_0-\eta.
	\end{array} \right.
	\end{align*} 
	Therefore, we may select $ w_n\in U $ such that $ s_n g(w_n) + t_n m(w_n) + 1 = 0 $.
	Similarly, we may select $ w^* \in U$
	such that $ u^* g(w^*) + t^* s(w^*) + 1 = 0 $.
	
	For each $ (s, t) \in R_n, $ $ \|(s - s^*, t - t^*) \| \leq \sqrt{2} 2^{-n} d$. Thus, we have
	$$  \vert s g^* + t m^* + 1\vert = \vert(s- s^*)g^* +(t - t^*)m^* \vert \leq \sqrt{2} 2^{-n} d \sqrt{(g^*)^2+(m^*)^2}:=M2^{-n}. $$
	If there are infinitely many $n$ such that $H(R_n)\ge H(R_{n-1})/2$ and suppose $n_k$ is the $k$th such $n$, then as $k\to\infty$,
	\begin{equation*}
	\begin{split}
	\int (s g^*+tm^*+1)^{-2}\mathrm{d}H(s,t)&\geq \int_{R_{n_k}}(sg^*+tm^*+1)^{-2}\mathrm{d}H(u, t)\\
	&\geq  H( R_{n_k}) M^{2} 2^{2n_k} \\
	&\geq 2^k4^{-n_k}\alpha M^2 2^{2n_k}\\
	&\geq M^2  \alpha 2^{k}\to \infty.
	\end{split}
	\end{equation*} 	
	Otherwise, there is an $N_0$ such that for all $n\ge N_0$, $H(R_{n-1}-R_n)\ge H(R_n)$. Therefore, 
	\begin{equation*}
	\begin{split}
	\int (s g^*+tm^*+1)^{-2}\mathrm{d}H(s,t)&\geq\sum_{n=1}^{\infty}\int_{R_{n-1}-R_n}(sg^*+tm^*+1)^{-2}\mathrm{d}H(s, t)\\
	&\geq \sum_{n=1}^{\infty} H(R_{n-1} - R_n) M^{2} 2^{2n} \\
	&\geq M^2 \sum_{n=N_0+1}^{\infty}H(R_n) 2^{2n}\\
	&\geq M^2 \sum_{n=N_0+1}^{\infty}4^{-n} \alpha 2^{2n}=\infty.
	\end{split}
	\end{equation*} 
	The proof is complete.
\end{proof}

\section{Appendix \label{app}}

In this section we do not assume any restrictions on $H$, and $c$ imposed in Theorem \ref{th77}.

We begin with establishing some results on sequences of probability distribution functions.   
It centers on the L\'evy distance $L(F,G)$
between two distribution functions defined as
$$L(F,G)=\inf\{\delta:G(x-\delta)-\delta\leq F(x)\leq G(x+\delta)+\delta \quad\text{for all x}\}.$$
It is a metric on the set of all distribution functions yielding weak convergence:  $G_n$ converging weakly to $G$ if and only if $d(G_n,G)\to0$.
It follows that
$$L(F,G)\leq\|F-G\|,$$
where $\|\cdot\|$ is the sup norm on functions.  Also, from Corollary A.42 and Theorem A.44 of  \cite{bai2010spectral} we have, if $A$ and $B$ are both $n\times N$, then
\begin{equation}\label{Levymatrix}
L^4(F^{AA^*},F^{BB^*})\leq\frac2{n^2}(\mathrm{tr}(AA^*+BB^*))(\mathrm{tr}[(A-B)(A-B)^*]),
\end{equation}
and
\begin{equation}\label{rankinequality}
\|F^{AA^*}-F^{BB^*}\|\leq\frac1n\rank(A-B).
\end{equation}
Since the rank of a matrix $A$ is equal to the dimension of its row space, we have
\begin{equation}\label{ranknumber}
\rank(A)\leq\text{ the number of nonzero entries of $A$}.
\end{equation}
The following rank inequalities are well-known:
For matrices $A$ and $B$ of the same dimensions
\begin{equation}\label{ranksum}
\rank(A+B)\leq\rank(A)+\rank(B).
\end{equation}
For matrices $A$ and $B$, where the product $AB$ is defined, we have
\begin{equation}\label{rankprod}
\rank(AB)\leq\min(\rank(A),\rank(B)).
\end{equation}
We will also need to consider the fact that, for Hermitian matrices $A$ and $B$, with $B$ being nonnegative definite,
\begin{equation}\label{traceprod}
\vert \mathrm{tr} AB\vert \leq\|A\|\mathrm{tr} B,
\end{equation}
where $\|\cdot\|$ is the spectral norm.

The following extends Lemmas 4.7 and 4.8 of \cite{bai2010spectral}.

\begin{lemma}\label{tight}
	Let $\{F_n\}$ be a sequence of distribution functions, and for each $\epsilon>0$ there exists a tight sequence of distribution
	functions $\{F_{n,\epsilon}\}$ such that for each $\delta>0$ there exists $\epsilon=\epsilon(\delta)$  for which
	\begin{equation}\label{levylimsup}
	\limsup_nL(F_{n,\epsilon},F_n)<\delta.
	\end{equation}
	Then the sequence $\{F_n\}$  is tight.  Moreover, if for each $\epsilon$, $F_{n,\epsilon}$ converges weakly to $F_{\epsilon}$, then $F_n$ converges
	weakly to $F$ with $F_{\epsilon}=F_{\epsilon(\delta)}$ converging weakly to $F$ as $\delta\to0$.
\end{lemma}	
\begin{proof}[Proof of Lemma \ref{tight}]
	 For any $\delta$, we have $L(F_{n,\epsilon(\delta/2)}, F_n)<\delta/2$ for all $n$ large.  
	 Choose $x$ so that $F_{n,\epsilon(\delta/2)}(x) > 1- \delta/2$ for all $n$.  
	Then, for all $n$ large
	$$1-\delta/2<F_{n,\epsilon(\delta/2)}(x)\leq F_n(x+\delta/2)+\delta/2.$$
	Therefore, $1-\delta< F_n(x+\delta/2)$ and since $\delta$ is arbitrary, we see that $\{F_n\}$ is tight.
	
	If $F_{n,\epsilon}$ converges weakly to $F_{\epsilon}$, suppose $F_1$ and $F_2$ are two distributions, each being weak limits of $F_n$ along two
	different subsequences $\{n'\}$, $\{n''\}.$  
	Then, for any $\epsilon$,
	\begin{align*}
	L(F_1,F_2)=&\lim_{n',n"\to\infty}L(F_{n'},F_{n''}) 
	\le \limsup_{n'\to\infty} L(F_{n'},F_{n',\epsilon})\\& + 
	\limsup_{n''\to\infty}L(F_{n''},F_{n'',\epsilon})+\lim_{n'\to\infty}L(F_{\epsilon},F_{n',\epsilon})+\limsup_{n''\to\infty}L(F_{\epsilon},F_{n'',\epsilon})\\
	=&\limsup_{n'\to\infty} L(F_{n'},F_{n',\epsilon})+\limsup_{n''\to\infty}L(F_{n''},F_{n'',\epsilon}),
	\end{align*}
	so it is evident that $L(F_1,F_2)$ can be made arbitrarily small.  This implies the weak convergence of $F_n$ to some distribution function $F$.
	Finally, we observe that
	\begin{align*}
		L(F,F_{\epsilon})=&\lim_{n\to\infty}L(F_n,F_{\epsilon})\leq\lim_{n\to\infty}L(F_n,F_{n,\epsilon})
		+\lim_{n\to\infty}L(F_{n,\epsilon},F_{\epsilon})\\
		=&\lim_{n\to\infty}L(F_{n},F_{n,\epsilon}),
	\end{align*}
	which can be made arbitrarily small. The proof of this lemma completes.
\end{proof}

Then we begin by truncating $R_n$ and $T_n$ to matrices of bounded norm for all $n$.
Since $(1/N)R_nR_n^*$ and $T_n$ commute, there exists unitary $\bbU_n$ which simultaneously diagonalizes these two matrices: 
\begin{align*}
	(1/N)R_nR_n^*= \bbU_n{\rm diag}(s_1,\ldots,s_n)\bbU_n^*,\quad T_n=\bbU_n{\rm diag}(t_1,\ldots,t_n)\bbU_n^*.
\end{align*} 
Write $R_n=\sqrt N\bbU_n{\rm diag}(s_1^{1/2},\ldots,s_n^{1/2})\bbV_n$, where $\bbV_n$ is $N\times N$ unitary, and 
${\rm diag}(s_1^{1/2},\ldots,s_n^{1/2})$ is $n\times N$ diagonal.
Since $H_n$ converges weakly to  $H$, for any $\varepsilon>0$, there exists a constant $\tau=\tau_{\epsilon}>0$ such that the number of eigenvalues of $(1/N)R_nR_n^*$ 
larger than 
$\tau$  and the number eigenvalues of $\bbT_n$ larger than $\tau$ are both less than $n\varepsilon/6$. Let
$$R_{n,\varepsilon}=\sqrt N\bbU_n{\rm diag}((\min(s_1,\tau))^{1/2},\ldots,(\min(s_n,\tau))^{1/2})\bbV_n,$$ 
$$T^{1/2}_{n,\varepsilon}=\bbU_n{\rm diag}((\min(t_1,\tau))^{1/2},\ldots,(\min(t_n,\tau))^{1/2})\bbU_n^*,$$ 
and define
$$C_{n,\varepsilon}=(1/N)T_{n,\varepsilon}^{1/2}(R_{n,\varepsilon}+X_n)(R_{n,\varepsilon}+X_n)^*T_{n,\varepsilon}^{1/2},$$
$$B_{n,\varepsilon}=(1/N)(R_{n,\varepsilon}+T_{n,\varepsilon}^{1/2}X_n)(R_{n,\varepsilon}+T_{n,\varepsilon}^{1/2}X_n)^*.$$
Then, using \eqref{rankinequality}, \eqref{ranksum}, and \eqref{rankprod}, we have
\begin{align}\label{rankbound}
&\max(\| F^{C_n}-F^{C_{n,\varepsilon}}\|,\| F^{B_n}-F^{B_{n,\varepsilon}}\|) \\ \nonumber
\le& \frac1n(\rank(R_n-R_{n,\varepsilon})+\rank(T_n-T_{n,\varepsilon}))\le \varepsilon/3 .
\end{align}

We turn now to truncating and centralizing the entries of $X_n$.   Since the Lindeberg condition hold for any $\eta>0$, we can find a sequence $\eta_n\to 0$ for which $\eta_n\sqrt n\to\infty$ and
\begin{align}\label{lindebergtrunc}
\frac1{nN\eta_n^2}\sum_{i=1}^n\sum_{j=1}^N\bbE\vert x_{ij}^2\vert I(\vert x_{ij}\vert>\eta_n\sqrt{n})\to 0.
\end{align}
Let $\hat x_{ij}=x_{ij}I(\vert x_{ij}\vert <\eta_n\sqrt{n})$ and construct $\widehat C_{n,\varepsilon}$,  $\widehat B_{n,\varepsilon}$ similarly as $C_{n,\varepsilon}$, $B_{n,\varepsilon}$ with $x_{ij}$ replaced by $\hat x_{ij}$.

Using \eqref{rankinequality}, \eqref{ranksum} and \eqref{rankprod}, we have
\begin{align*}
&\max(\|F^{\widehat C_{n,\varepsilon}}-F^{C_{n,\varepsilon}}\|,\|F^{\widehat B_{n,\varepsilon}} 
-F^{B_{n,\varepsilon}}\|) \\ 
\le & \frac1n\rank(X_n-\widehat X_n) \le\frac1n\sum_{i=1}^n\sum_{j=1}^N 
I(\vert x_{ij}\vert \ge\eta_n\sqrt{n}).
\end{align*}
Note that
\begin{align*}
& \rVar\left(\sum_{i=1}^n\sum_{j=1}^NI(\vert x_{ij}\vert\ge\eta_n\sqrt{n}) \right)\le \bbE\left(\sum_{i=1}^n\sum_{j=1}^NI(\vert x_{ij}\vert \ge\eta_n\sqrt{n}) \right)\\
\le &\frac1{n\eta_n^2}\sum_{i=1}^p\sum_{j=1}^n\bbE\vert x_{ij}\vert^2 I(\vert x_{ij}\vert \ge\eta_n\sqrt{n}) =o(n).
\end{align*}
Therefore, by Bernstein inequality (p. 21 of \cite{bai2010spectral}), we have
\begin{align}\label{t2}
&\Pr \left(\max(\|F^{\widehat C_{n,\varepsilon}}-F^{C_{n,\varepsilon}}\|\|,F^{\widehat B_{n,\varepsilon}}-F^{B_{n,\varepsilon}}\|)\ge \varepsilon/6\right) \\ \nonumber
\le & 2\exp\left(\frac{-(\varepsilon n/12)^2}{o(n)+\varepsilon n/12}\right)\le 2e^{-\delta n},
\end{align}
where $\delta$ is some positive constant. The right hand side is summable, hence, 
\begin{align}\label{t3}
\max(\|F^{\widehat C_{n,\varepsilon}}-F^{C_{n,\varepsilon}}\|, \|F^{\widehat B_{n,\varepsilon}}-F^{B_{n,\varepsilon}}\|)\le \varepsilon/6, a.s. \mbox{ for all large } n.
\end{align}
Define
\begin{align*}
\tilde x_{ij}=\begin{cases}(\hat x_{ij}-\bbE \hat x_{ij})/\sigma_{ij}, & \mbox{ if } \sigma^2_{ij}=\bbE \vert \hat x_{ij}-\bbE \hat x_{ij}\vert^2\ge 1/2,\\ 
y_{ij},& \mbox{ otherwise,}
\end{cases}
\end{align*}
where $y_{ij}$ are i.i.d. random variables taking values $\pm 1$ with equal probabilities. Define $\widetilde C_n$ and $\widetilde B_n$ 
with $\tilde x_{ij}$ replacing $\hat x_{ij}$. Then, from \eqref{Levymatrix} and \eqref{traceprod},
we have
\begin{align}\label{t4}
L^4(F^{\widehat C_{n,\ep}},F^{\widetilde C_{n,\ep}})\le \frac{2\tau}{n^2N}(\rtr\widehat C_{n,\ep}+\rtr\widetilde C_{n,\ep})\rtr (\widehat X_n-\widetilde X_n)(\widehat X_n-\widetilde X_n)^*,
\end{align}
and
\begin{align*}
L^4(F^{\widehat B_{n,\ep}},F^{\widetilde B_{n,\ep}})\le \frac{2\tau}{n^2N}(\rtr\widehat B_{n,\ep}+\rtr\widetilde B_{n,\ep})\rtr (\widehat X_n-\widetilde X_n)(\widehat X_n-\widetilde X_n)^*.
\end{align*}
Using \eqref{traceprod} and the fact that $(a+b)^2\leq 2(a^2+b^2)$, we have         
\begin{align*}
n^{-1}\rtr\widehat C_{n,\ep}\le 2\tau(nN)^{-1}\rtr(R_{n,\ep}R_{n,\ep}^*+\widehat X_{n}\widehat X_{n,}^*),
\end{align*}
\begin{align*}
n^{-1}\rtr\widetilde C_{n,\ep}\le 2\tau(nN)^{-1}\rtr(R_{n,\ep}R_{n,\ep}^*+\widetilde X_{n}\widetilde X_{n}^*),
\end{align*}
\begin{align*}
n^{-1}\rtr\widehat B_{n,\ep}\le 2(nN)^{-1}\rtr(R_{n,\ep}R_{n,\ep}^*+\tau\widehat X_{n}\widehat X_{n}^*),
\end{align*}
\begin{align*}
n^{-1}\rtr\widetilde B_{n,\ep}\le 2(nN)^{-1}\rtr(R_{n,\ep}R_{n,\ep}^*+\tau\widetilde X_{n}\widetilde X_{n}^*).
\end{align*}

We have $(nN)^{-1}\rtr R_{n,\ep}R_{n,\ep}^*\leq\tau$ and claim that both $(nN)^{-1}\rtr\widehat X_{n}\widehat X_{n}^*$ and
$(nN)^{-1}\rtr\widetilde X_{n}\widetilde X_{n}^*$ converge a.s. to 1. To support this claim, we note that
$$(nN)^{-1}\bbE\rtr\widehat X_{n}\widehat X_{n}^*=1- (nN)^{-1}\sum_{ij}\bbE\vert x_{ij}\vert^2I(\vert x_{ij}\vert \ge \eta_n\sqrt{n})\to1.$$
Since $\bbE\vert x_{ij}\vert^kI(\vert x_{ij}\vert<\delta_n\sqrt n)\leq n^{\frac{k-2}2}$ for $k\ge2$, the fourth central moment of each quantity is bounded by $Kn^{-8}(n^2n^3+n^4n^2),$ 
which is summable. The claim is proven.

We have 
\begin{align*}
&\frac1{nN}\rtr (\widehat X_n-\widetilde X_n)(\widehat X_n-\widetilde X_n)^*
=\frac1{nN}\sum_{i=1}^n\sum_{j=1}^N\vert\hat x_{ij}-\tilde x_{ij}\vert^2:=I_1+I_2,
\end{align*}
where
\begin{align*}
I_1=\frac1{nN}\sum_{\gs^2_{ij}\ge 1/2}\vert \hat x_{ij}-\tilde x_{ij}\vert^2,\\
I_2=\frac1{nN}\sum_{\gs^2_{ij}< 1/2}\vert \hat x_{ij}- y_{ij}\vert^2.
\end{align*}
We see that 
\begin{align*}
&\bbE I_1=\frac1{nN}\sum_{\gs^2_{ij}\ge 1/2}\bbE\left\vert \frac{\bbE x_{ij}I(\vert x_{ij}\vert \ge\eta_n\sqrt n)}{\gs_{ij}}-\hat x_{ij}\frac{\gs_{ij}-1}{\gs_{ij}}\right\vert^2\\
\le&\frac2{nN}\sum_{\gs^2_{ij}\ge 1/2}\left(\frac1{\eta_n^2n}+\left(\gs_{ij}-1\right)^2\right)\\ \le&\frac2{nN}\sum_{\gs^2_{ij}\ge 1/2}\bigg(\frac1{\eta_n^2n}+\bigg(\bbE\vert x_{ij}\vert^2I(\vert x_{ij}\vert \leq \eta_n\sqrt n)-1 - \vert\bbE x_{ij}I(\vert x_{ij}\vert \leq \eta_n\sqrt n)\vert^2 \bigg)^2\bigg)\\
=&\frac2{nN}\sum_{\gs^2_{ij}\ge 1/2}\left(\frac1{\eta_n^2n}+\left(\bbE\vert x_{ij}\vert^2I(\vert x_{ij}\vert>\eta_n\sqrt n)+\vert \bbE x_{ij}I(\vert x_{ij}\vert>\eta_n\sqrt n)\vert^2\right)^2\right)
\\& \le\frac2{nN}\sum_{\gs^2_{ij} \ge 1/2}\left(\frac1{\eta_n^2n}+4\bbE\vert x_{ij}\vert^2I(\vert x_{ij}\vert>\eta_n\sqrt n)\right)\to 0.
\end{align*}
And
\begin{align*}
\bbE\vert I_1-\bbE I_1\vert^4 
&\leq \frac K{n^8}\bigg(\sum_{ij}\left\vert \vert \hat x_{ij}-\tilde x_{ij}\vert^2-\bbE\vert\hat x_{ij}-\tilde x_{ij}\vert^2\right\vert^4\\
&+\sum_{\{i,j\}\neq\{i',j'\}}
\bbE\left\vert \vert\hat x_{ij}-\tilde x_{ij}\vert^2-\bbE\vert \hat x_{ij}-\tilde x_{ij}\vert^2\right\vert^2 \\& \quad \quad \quad \bbE\left\vert\vert\hat x_{i'j'}-\tilde x_{i'j'}\vert^2-\bbE\vert \hat x_{i'j'} -\tilde x_{i,j'}\vert^2\right\vert^2\bigg)\\&
\le \frac K{n^8}(nN\times n^3 + (nN)^2n^2),
\end{align*}
which is summable and hence
$I_1\to 0, a.s.$
When $\sigma^2_{ij}<1/2$, we have
\begin{align*}
1/2>&\bbE\vert x_{ij}\vert^2I(\vert x_{ij}\vert\leq\eta_n\sqrt n)-\vert\bbE x_{ij}I(\vert x_{ij}\vert \leq\eta_n\sqrt n)\vert^2\\&
=1-\bbE\vert x_{ij}\vert^2I(\vert x_{ij}\vert>\eta_n\sqrt n)-\vert \bbE x_{ij}I(\vert x_{ij}\vert>\eta_n\sqrt n)\vert^2,
\end{align*}
which implies
$$\bbE \vert x_{ij}^2\vert I(\vert x_{ij}\vert \ge\eta_n\sqrt{n})\ge \frac12-o(1)\ge \frac13,$$
for all $n$ large.
Denote by $N$, the number of pairs $(i,j)$ such that $\gs_{ij}^2<1/2$. 
When $\gs_{ij}^2<1/2$, we have 
\bqn \sum_{\gs_{ij}^2<1/2}\bbE \vert x_{ij}^2\vert I(\vert x_{ij}\vert \ge\eta_n\sqrt{n})\ge N/3.
\eqn
Upon combining the above with the relation,
\bqn \sum_{\gs_{ij}^2<1/2}\bbE \vert x_{ij}^2\vert I(\vert x_{ij}\vert \ge\eta_n\sqrt{n})\le o(nN),
\eqn
it becomes evident that for any arbitrary $\epsilon>0$, it holds almost surely that 
$$\limsup_n\max(L(C_{n,\epsilon},\widetilde C_{n,\epsilon}),L(B_{n,\epsilon},\widetilde B_{n,\epsilon}))\leq\epsilon/3.$$

The results in \cite{zhou2022limiting} can now be applied to $C_n$ and $C_{n,\epsilon}$.   
Let $g_{n,\epsilon}(z)=(1/n)\rtr (C_{n,\epsilon}-zI)^{-1}T_{n,\epsilon}$.    
Assume $(x,y)$ is a random vector with joint distribution function $H(s,t)$,  denote by $H^{\epsilon}$ the joint distribution function of $(\min(x,\tau),\min(y,\tau))$.   
Then from \cite{zhou2022limiting}, we have for all positive $c$ and $z\in\mathbb{C}^+$, almost surely, for any converging subsequence of $(m_{C_{n,\epsilon}}(z)$,$g_{n,\epsilon}(z))$ to $(m,g)$, this pair must satisfy \eqref{first} with $H$ replaced by  $H^{\epsilon}$.  It amounts to verify that there is only one  $m$ satisfying the equations for an infinite number of $z\in\mathbb{C}^+$ with an accumulation point.

Now we turn our attention to $B_n$, the definition of which also incorporates $C_n$.   
The transition from \eqref{first} to \eqref{second} persists when $T_n$ remains invertible.
To this end, we define (without loss of generality we can assume $\epsilon<\tau$)
$$\widetilde T^{1/2}_{n,\ep}=\bbU_n{\rm diag}((\max(\epsilon,\min(t_1,\tau)))^{1/2},\ldots,(\max(\epsilon,\min(t_n,\tau)))^{1/2})\bbU_n^*,$$ 
and 
$$\widetilde{\widetilde B}_{n,\ep}=(1/N)(R_{n,\ep}+\widetilde T_{n,\ep}^{1/2}\widetilde X_n)(R_{n,\ep}+\widetilde T_{n,\ep}^{1/2}\widetilde X_n)^*.$$ 
Then from \eqref{Levymatrix} and \eqref{traceprod}, we have
\begin{align*}
L^4(F^{\widetilde{\widetilde B}_{n,\ep}},F^{\widetilde B_{n,\ep}})&\le \frac2{Nn^2}(\rtr \widetilde{\widetilde B}_{n,\ep}+\rtr\widetilde B_{n,\ep})\rtr (T_{n,\ep}^{1/2}-\widetilde T_{n,\ep}^{1/2})\widetilde X_n\widetilde X_n^*(T_{n,\ep}^{1/2}-\widetilde T_{n,\ep}^{1/2}) \\&
\le \frac{K\epsilon}{nN}\rtr \widetilde X_n\widetilde X_n^*\to K\epsilon, a.s.
\end{align*}

We see that, almost surely, the assumption \eqref{levylimsup} in Lemma \ref{tight} is met.
In order to demonstrate uniqueness, we shift our focus to the matrix $\underline{B}_n$, as defined in \eqref{Bnul}. 
As noted earlier, the eigenvalues of $B_n$ are the same as those of $\underline B_n$, except for $ \lvert n-N\rvert $ zero eigenvalues. 
The relation between their ESDs and Stieltjes transforms is given in \eqref{BBnul}.
After making the variable transformations in \eqref{vartrans} we arrive at the equations in \eqref{third}.

Fix $z=x+iv$, $v>0$. Notice that, when $ (m,g) $ are finite limits of $(m^{B_n}(z),g_n(z))$ then necessarily 
$\Im \underline{m},\ \Im\underline{g},\ \Im (z\underline{m}),\ \Im (z\underline{g})$ are all nonnegative.
Suppose $ (\underline{m}, \underline{g}) $ is such a set of solutions to \eqref{third}.  
Recall the identity \eqref{7222} in section 2. In the case where $H$ has bounded support, it becomes apparent that we can make $B_0$ approach 1 as closely as desired, hence the quantity $(1-c)v+cB_0v$ positive, for suitably small $\underline m$ and $\underline g$. Therefore for these values of  $\underline m$ and $\underline g$, we have \eqref{Vlessthanz}.
Redefining $B_{n,\epsilon}$ to be
$B_{n,\epsilon}=(1/N)(R_{n,\epsilon}+\widetilde T_{n,\epsilon}X_n)(R_{n,\epsilon}+\widetilde T_{n,\epsilon}X_n)^*$, we use the results on $C_{n,\epsilon}$ to
conclude that, almost surely, for any weakly convergent subsequence of $F^{B_{n,\epsilon}}$, the corresponding $m$ will satisfy \eqref{second} for some $g$, and therefore the corresponding $(\underline m,\underline g)$ will satisfy \eqref{third}.

Since $\widetilde T_{n,\epsilon}$ has bounded spectral norm for all $n$, we see that, with $g_{n,\epsilon}$ now equal to $(1/n)\rtr(B_{n,\epsilon}-zI)^{-1}\widetilde T_{n,\epsilon}$,
$$\frac1{\frac1n\rtr\widetilde T_{n,\epsilon}}g_{n,\epsilon}(z)$$
is the Stieltjes transform of a probability measure.  Consequently, from \eqref{vartrans},
$$\underline g_{n,\epsilon}(z)\equiv-\frac1{z(1+c_ng_{n,\epsilon}(z))}$$
satisfies the conditions of being the Stieltjes transform of a probability measure $P_{n,\epsilon}$ with mass on the nonnegative reals.  Any vaguely convergent subsequence of $P_{n,\epsilon}$ will have
the limiting Stieltjes transform also of this form, which is the Stieltjes transform of a probability measure concentrated on the nonnegative reals.  Therefore the convergence is weak.

Suppose, with probability one, there is a subsequence for which both $F^{B_{n,\epsilon}}$ and $P_{n,\epsilon}$ converge weakly, with resulting limiting Stieltjes transforms $\underline m(z)$ and $\underline g(z)$, which satisfy \eqref{third} for all $z=x+iv\in\mathbb{C}^+$.   Since Stieltjes transforms of probability measures are bounded by $1/v$, we can find $v$ suitably large so that \eqref{Vlessthanz} holds.  
Suppose for one of this $z$, there are two  sets of solutions $ (\underline{m}_{(i)}, \underline{g}_{(i)}) $, $ i = 1, 2 $, resulting, almost surely, from two weakly converging subsequences of $F^{B_{n,\epsilon}}$ and $P_{n,\epsilon}$ with $ \um_{(1)}\ne\um_{(2)}$.  
We have then \eqref{7228}, except there is a strict inequality at the last step, resulting in a contradiction.
Thus we have  unique Stieltjes transforms for any almost surely weakly converging subsequence of $F^{B_{n,\epsilon}}$ and $P_{n,\epsilon}$. So that, with probability one, $F^{B_{n,\epsilon}}$ and $P_{n,\epsilon}$ converge weakly to a nonrandom probability distribution function $F^{\epsilon}$ and a 
nonrandom probabilty measure $P^{\epsilon}$ with limiting Stieltjes transforms satisfying \eqref{second}.  
From Lemma \ref{tight} we have almost surely $F^{B_n}$ converging weakly to a nonrandom distribution function $F$, which implies that $F^{C_n}$ also converges almost surely weakly to a nonrandom distribution function.

It amounts to show that this limiting distribution satisfies the equations.  
With $x,y$ denoting random variables with joint distribution function $H(s,t)$, we define $H^{\epsilon}(s,t)$ as the joint distribution function of the random variables $\min(x,\tau_{\epsilon})$ and $\max(\epsilon,\min(y,\tau_{\epsilon}))$. And let $F^{\epsilon}$ be the distribution function associated with $H^{\epsilon}$. 
Then as $\epsilon\to0$, $F^{\epsilon}$ and $H^{\epsilon}$ converge in distribution to $F$ and $H$, respectively.   
Fix $z=x+iv\in\mathbb{C}^+$.  With $m^{\epsilon}=m^{\epsilon}(z)$, $\underline m^{\epsilon}=\underline m^{\epsilon}(z)$,
$g^{\epsilon}=g^{\epsilon}(z)$, $\underline g^{\epsilon}=\underline g^{\epsilon}(z)$ denoting the values in equations \eqref{second} and \eqref{third},
since $F^{\epsilon}$ converges in distribution to $F$, we immediately get $m^{\epsilon}$ and $\underline m^{\epsilon}$ converging to $m$, the Stieltjes transform of $F$ at $z$, and to $\underline m\equiv-\frac{1-c}z +cm$, respectively. 

We claim that $g^{\epsilon}$ remains bounded as $\epsilon\to0$.   We have $\underline m_2$ the imaginary part of $\underline m$ is positive, since
$\underline m$ is the value of a Stieljes transform for $z\in\mathbb{C}^+$.
Let $\delta>0$ be a lower bound on $\underline m_2^{\epsilon}$.  Then the integrand in the second equation in \eqref{third} satisfies
$$\left\vert\frac{t}{1+s\underline g+t\underline m}\right\vert=\left\vert \frac{t}{1+s\underline g+t\underline m_1+it\underline m_2}\right\vert
\leq\frac1{\delta}.$$
Suppose on a sequence  $\epsilon_n$, $g^{\epsilon_n}$ goes unbounded.  Then necessarily $\underline g^{\epsilon_n}\to0$.  But from the second
equation of \eqref{third}, we see the right side goes unbounded, while the left side remains at $z$, a contradiction.   

On a sequence $\epsilon_n$, let $g=\lim_{n\to\infty}g^{\epsilon_n}$. Notice the integrand in the first equation in \eqref{second} is bounded in 
absolute value by $1/v$.   We have
\begin{align*}
&\int\frac{\mathrm{d}H^{\epsilon_n}(s,t)}{\frac{s}{1+cg^{\epsilon_n}}-(1+cm^{\epsilon_n})z+t(1-c)}-\int\frac{\mathrm{d}H(s,t)}{\frac{s}{1+cg}-(1+cm)z+t(1-c)}
\\=&\int\frac{\left(\frac{s}{1+cg}-\frac{s}{1+cg^{\epsilon_n}}+m^{\epsilon_n}-m\right)\mathrm{d}H^{\epsilon_n}(s,t)}{\left(\frac{s}{1+cg^{\epsilon_n}}-(1+cm^{\epsilon_n})z+t(1-c)\right)\left(\frac{s}{1+cg}-(1+cm)z+t(1-c)\right)}\label{firstinsecond}\\&
+\int\frac{\mathrm{d}H^{\epsilon_n}(s,t)}{\frac{s}{1+cg}-(1+cm)z+t(1-c)}-\int\frac{\mathrm{d}H(s,t)}{\frac{s}{1+cg}-(1+cm)z+t(1-c)} \\& \to 0
\quad\text{ as $n\to\infty$}.
\end{align*}

On the same sequence $\epsilon_n$, let $\ug=\lim_{n\to\infty}\ug^{\epsilon_n}$.   
The relationship between $g$ and $\ug$ in \eqref{vartrans} still stands. Considering that the absolute value of the denominator on the right side is at least $v$, it is clear that $\ug_2 >0$. Now, let $\delta>0$ be a lower bound for both $\um_2^{\epsilon}$ and $\ug^{\epsilon_n}_2$.
For the second equation in \eqref{third}, we have 
\begin{align*}
&\limsup_n\left\vert \int\frac{t\mathrm{d}H^{\epsilon_n}(s,t)}{1+s\underline g^{\epsilon_n}+t\underline m^{\epsilon_n}}-\int\frac{t\mathrm{d}H(s,t)}{1+s\underline g+t\underline m}\right\vert
\\ \le&\limsup_n\left\vert\int\frac{(s(\underline g-\underline g^{\epsilon_n})+t(\underline m-\underline m^{\epsilon_n}))t\mathrm{d}H^{\epsilon_n}(s,t)}{\left(1+s\underline g^{\epsilon_n}+t\underline m^{\epsilon}\right)\left(1+s\underline g+t\underline m\right)}\right\vert \\&
+\limsup_n\left\vert\int\frac{t\mathrm{d}H^{\epsilon}(s,t)}{1+s\underline g+t\underline m}-\int\frac{t\mathrm{d}H(s,t)}{1+s\underline g+t\underline m}\right\vert\\ \leq&\frac1{\delta^2}\limsup_n(\vert\underline g-\underline g^{\epsilon_n}\vert +\vert\underline m-\underline m^{\epsilon_n}\vert)=0.
\end{align*}

We conclude that for every $z\in\mathbb{C}^+$, the limiting $F$ has its Stieltjes transform $m=m(z)$ satisfying \eqref{second} for some $g$ with
$\Im g\ge0$. This is along with the corresponding $\underline m$, $\underline g $ satisfying \eqref{third}.

\section*{Acknowledgments}
The authors would like to thank the anonymous referee, the Associate Editor and the Editor for their invaluable and constructive comments.
J. Hu was supported by NSFC (No. 12171078, 11971097) and National Key R \& D Program of China (No. 2020YFA0714102).
Z. D. Bai was partially supported by NSFC Grant 12171198 and Team Project of Jilin Provincial Department of Science and Technology (No.20210101147JC).



\end{document}